\RequirePackage{fix-cm}
\documentclass[twocolumn]{svjour3}

\smartqed  % flush right qed marks, e.g. at end of proof

\usepackage{amsmath}
\usepackage{amssymb}
\usepackage{amsthm}
\usepackage{graphicx}
\usepackage{mathrsfs}
\usepackage{hyperref}
\usepackage{color}
\usepackage{multirow}
\usepackage{array}
\usepackage{booktabs}
\usepackage{enumitem}   
\usepackage{mathabx}
\usepackage[misc,geometry]{ifsym}

\DeclareMathOperator\Span{span}
\DeclareMathOperator\spect{Spect}

\theoremstyle{definition}

\begin{document}

\title{Fast computation and characterization of forced response surfaces via spectral submanifolds and parameter continuation}
% Nonlinear Characterization of Mechanical Systems via Forced Response Surface and Damped Backbone Curves: Spectral Submanifolds and Parameter Continuation
\titlerunning{Fast computation and characterization of forced response surface}

\author{Mingwu Li         \and
        Shobhit Jain      \and
        George Haller
}

%\authorrunning{Short form of author list} % if too long for running head

\institute{M. Li (\Letter) \at
              Department of Mechanics and Aerospace Engineering\\ Southern University of Science and Technology \\
              Shenzhen, 518055, China\\
              \email{limw@sustech.edu.cn} 
           \and
           S. Jain \at
              Delft Institute of Applied Mathematics, TU Delft\\
              Mekelweg 4, 2628CD, Delft, The Netherlands
           \and
           G. Haller \at
              Institute for Mechanical Systems, ETH Z\"{u}rich \\
              Leonhardstrasse 21, 8092 Zurich, Switzerland          
}

\date{Received: date / Accepted: date}

\maketitle

\begin{sloppypar}
\begin{abstract}
     For mechanical systems subject to periodic excitation, forced response curves (FRCs) depict the relationship between the amplitude of the periodic response and the forcing frequency. For nonlinear systems, this functional relationship is different for different forcing amplitudes. Forced response surfaces (FRSs), which relate the response amplitude to both forcing frequency and forcing amplitude, are then required in such settings. Yet, FRSs have been rarely computed in the literature due to the higher numerical effort they require. 
     Here, we use spectral submanifolds (SSMs) to construct reduced-order models (ROMs) for high-dimensional mechanical systems and then use multidimensional manifold continuation of fixed points of the SSM-based ROMs to efficiently extract the FRSs.
     Ridges and trenches in an FRS characterize the main features of the forced response. We show how to extract these ridges and trenches directly without computing the FRS via reduced optimization problems on the ROMs. We demonstrate the effectiveness and efficiency of the proposed approach by calculating the FRSs and their ridges and trenches for a plate with a 1:1 internal resonance and for a shallow shell with a 1:2 internal resonance.
\end{abstract}
\end{sloppypar}


\section{Introduction}

\begin{sloppypar} 
Forced response curves (FRCs) are important constructs for understanding forced nonlinear response in mechanical systems. For a mechanical system subject to periodic excitation at a given forcing amplitude, an FRC depicts the functional relationship between the amplitude of the periodic response and the forcing frequency. FRCs offer various practical insights for nonlinear systems, especially for those with internal resonances~\cite{nayfeh1988undesirable,balachandran1991observations,antonio2012frequency}.

For nonlinear systems, FRCs constructed at different forcing amplitudes can differ even qualitatively, let alone quantitatively. Indeed, the FRCs of a damped nonlinear system subject to periodic forcing at low amplitudes resemble the linearized periodic response, where we obtain a one-to-one relationship between the response amplitude and the forcing frequency. However, for moderate or high forcing levels, the FRC can deviate qualitatively from its linearized counterpart, featuring even multiple periodic solutions for a given forcing frequency ~\cite{Nayfeh1995}. To account for such nonlinear dependence of the FRCs on forcing amplitudes,  forced response surfaces (FRSs) are of great significance. 

An FRS is a two-dimensional surface that depicts the relationship of the response amplitude to both the forcing frequency and the forcing amplitude. Hence, an FRS provides a complete dynamic characterization of nonlinear systems subject to periodic forcing.  In this work, we exploit spectral submanifolds (SSMs) and parameter continuation for the efficient computation of the FRSs in forced, damped, nonlinear mechanical systems that may also be internally resonant. 

\subsection{Forced response surface}

An FRS can be interpreted as a one-parameter family of FRCs with varying forcing amplitudes. However, constructing an FRS from a collection of FRCs is challenging for the following two reasons. First, the prediction of any potential isolas in an FRC is difficult. An isola is an isolated branch of periodic orbits that is detached from the main branch~\cite{ponsioen2019analytic,2012On}. An FRC is typically obtained via one-dimensional parameter continuation that faces difficulties in locating isolas. This is because continuation along an isola requires initial conditions close to the detached solution, whose location is a priori unknown~\cite{ponsioen2019analytic}. The second challenge arises in sampling the forcing amplitudes so that all nontrivial features in the FRS can be reconstructed. Such an amplitude sampling is problem-dependent and may require additional tuning.

The above two challenges can be overcome by directly computing the FRS via two-dimensional parameter continuation. As any isola merges with the main branch for sufficiently large forcing amplitudes~\cite{ponsioen2019analytic}, the birth and disappearance of isolas can be automatically detected by computing the FRS as a single two-dimensional object~\cite{2012On}. Furthermore, the sampling of forcing amplitudes is also not required, as we detail below. 

An FRS is a two-dimensional surface that covers the periodic response amplitude under combined variations in the forcing frequency and amplitude. Standard multi-dimensional continuation algorithms, such as the Henderson algorithm~\cite{henderson2002multiple}, are useful for FRS computation. In the Henderson algorithm~\cite{henderson2002multiple}, any surface is approximated by a piecewise polyhedral tessellation. In our setting, this avoids any manual sampling of the forcing amplitude. This algorithm has been implemented in the software packages \textsc{multifario}~\cite{Multifario} and \textsc{coco}~\cite{COCO}, and has recently been extended to adaptive boundary-value problems~\cite{dankowicz2020multidimensional}.  Indeed, Henderson's algorithm has also been used to extract the FRS of coupled oscillators~\cite{2021Topology}.

% high-dimensionality - SSM reduction
For an efficient extraction of the FRS, fast computation of periodic orbits is necessary. Periodic orbits of low-dimensional nonlinear systems can be obtained via various numerical methods such as numerical integration, shooting methods~\cite{keller2018numerical,coco-shoot}, collocation schemes~\cite{dankowicz2013recipes} and harmonic balance techniques~\cite{krack2019harmonic}. However, the computational costs of these methods are prohibitive for high-dimensional systems such as finite element (FE) models~\cite{jain2022compute,part-i}. To reduce this computational cost, reduced-order models are paramount.

\subsection{Nonlinear model reduction via SSMs}
The recent theory of spectral submanifolds~\cite{haller2016nonlinear} (SSM) has enabled rigorous model reduction of nonlinear mechanical systems. SSMs are invariant manifolds that serve as unique non-linear continuations of modal subspaces for damped nonlinear systems. Furthermore, SSMs are guaranteed to exist when appropriate non-resonance conditions are satisfied on the eigenvalues of the linearization~\cite{haller2016nonlinear}. With SSM reduction, periodic orbits of high-dimensional systems appear as fixed points of low-dimensional SSM-based reduced-order models (ROMs)~\cite{breunung2018explicit,jain2022compute,part-i}. In particular, analytic prediction of FRCs via two-dimensional SSM-based ROMs is possible if the full system admits no internal resonances~\cite{jain2022compute}. In this work, we further demonstrate that the FRS associated with two-dimensional SSMs can also be obtained analytically. For systems with internal resonances, higher-dimensional SSMs are relevant for model reduction~\cite{part-i,part-ii} and the reduced dynamics on such SSMs can again be used to compute the FRS via the aforementioned multi-dimensional continuation algorithms.

% further speedup - ridges and trenches,
% link to damped backbone curves - relations and extensions - extended damped backbone curves 
As we will discuss, the ridges and trenches of an FRS characterize the main features of the forced response. In fact, a ridge generalizes the notion of a damped backbone curve, which is obtained by connecting the points of maximal response amplitude in the FRCs at various forcing amplitudes~\cite{breunung2018explicit}. For weakly damped systems, this damped backbone curve can be approximated via the force appropriation method~\cite{PEETERS2011,Peeters2011a} or the resonance decay method~\cite{szalai2017nonlinear,breunung2018explicit}. However, these procedures lose their validity for systems with moderate damping or internal resonances. Computing the FRS, on the other hand, is more generally valid for the nonlinear characterization of mechanical systems. 

As the ridges and trenches of an FRS delineate the main features of interest in the forced response, it is natural to ask if we can extract them directly without computing the entire FRS. Indeed, this is possible by formulating appropriate optimization problems for periodic orbits and then solving them via successive continuation techniques~\cite{kernevez1987optimization,li2018staged,li2020optimization}, as we will demonstrate. Computing the ridges and trenches would provide a quick characterization of the FRS as continuation along these curves will be faster than that along the two-dimensional FRS. More importantly, we will show that SSM-based model reduction will further reduce the optimization problem for periodic orbits to an optimization problem for fixed points, which is essential for fast extraction of ridges and trenches of the FRSs of high-dimensional systems.

% organization of the rest of the paper
The remainder of this paper is organized as follows. In the next section, we discuss the computation of the FRS and formulate the optimization problems that define the ridges and trenches of the FRS. We then review the SSM theory and show how SSM-based ROMs can be used for FRS computation. Next, we use SSM-based ROMs to simplify the optimization problems for computing ridges and trenches in the FRS. We further present a solution method for these simplified optimization problems via parameter continuation. Finally, we demonstrate the effectiveness of our procedure on various numerical examples before drawing conclusions.

\end{sloppypar}

%Another useful approach for the nonlinear characterization of forced response \cite{Detroux2015}

\section{Problem Formulation}

\subsection{A motivating example}
Consider a harmonically excited linear oscillator given as 
\begin{equation}
\label{eq:lin-osci}
    \ddot{x}+2\zeta\dot{x}+x=\epsilon\cos\Omega t,
\end{equation}
where $\zeta\in(0,1/\sqrt{2}]$ is a fixed damping coefficient, while the forcing amplitude $\epsilon$ and the forcing frequency $\Omega$ are free to change. This linear system admits a periodic solution in the form $x(t)=C\cos(\Omega t-\theta)$, where $C>0$ is the amplitude of the periodic response given as
\begin{equation}
\label{eq:C}
    C(\Omega,\epsilon) = \frac{\epsilon}{\sqrt{(1-\Omega^2)^2+4\zeta^2\Omega^2}},
\end{equation}
and $\theta$ is the phase lag with respect to the forcing. The expression~\eqref{eq:C}  for $C(\Omega,\epsilon)$ characterizes the forced response surface (FRS) of the linear oscillator under variations in $(\Omega,\epsilon)$. A visualization of this surface with $\zeta=0.1$ is shown in Fig.~\ref{fig:FRS-lin}, where we see that there is a ridge on the surface.

\begin{figure}[!ht]
\centering
\includegraphics[width=.45\textwidth]{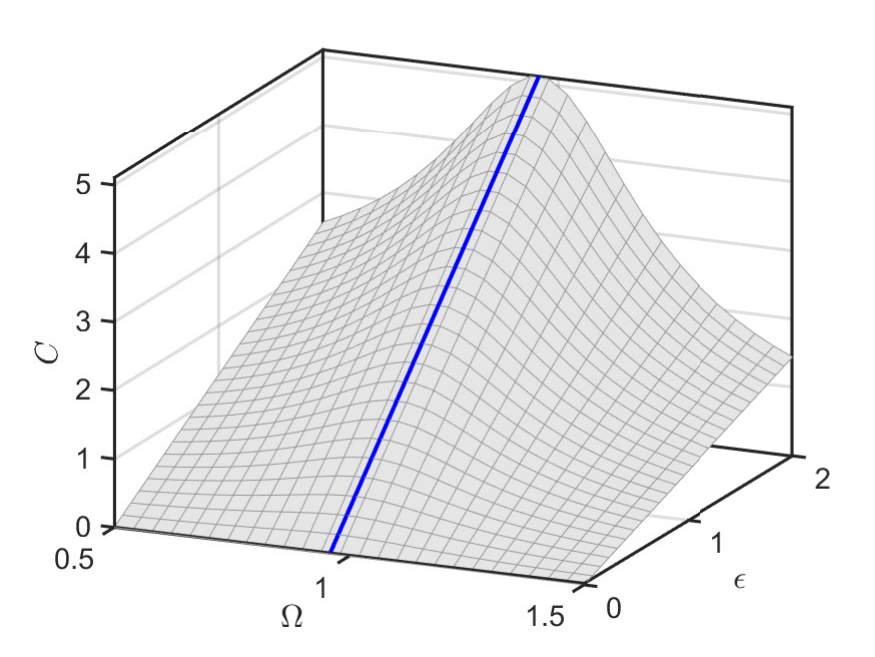}
\includegraphics[width=.45\textwidth]{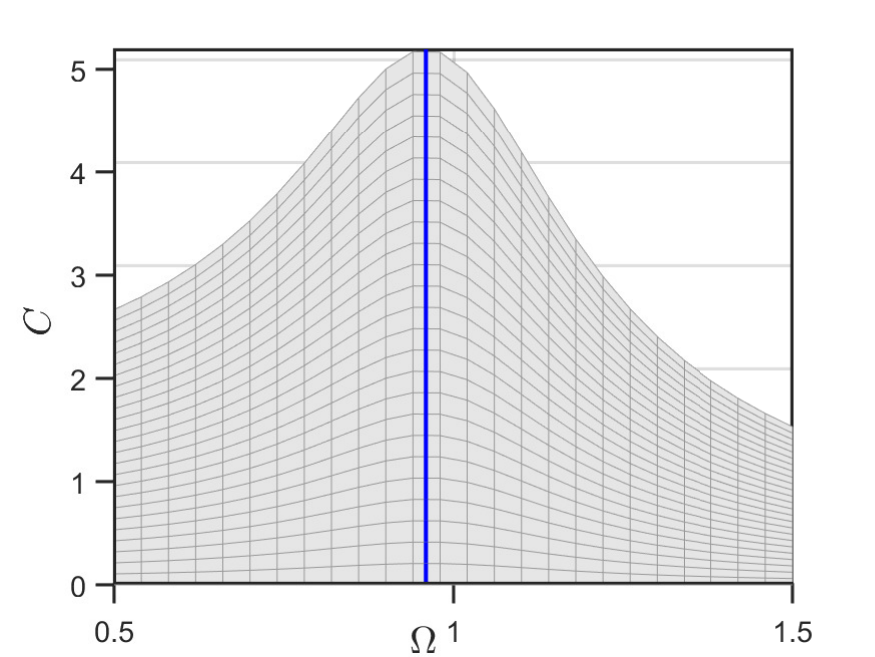}
\caption{\small Forced response surface of the harmonically forced linear oscillator~\eqref{eq:lin-osci} (upper panel) and its projection onto the plane $(\Omega,C)$ (lower panel). Here the surface plot is based on the explicit expression for $C$ in~\eqref{eq:C} with $\zeta=0.1$, the blue line is the ridge of the surface (based on~\eqref{eq:Cridge}), and the gray lines on the surface represent response curves with fixed $\Omega$ or $\epsilon$.}
\label{fig:FRS-lin}
\end{figure}

The ridge in Fig.~\ref{fig:FRS-lin} is a curve connecting the local maxima of the forced response curves (FRCs) under the variation in the forcing amplitude $\epsilon$. In particular, taking a section in the graph of $C$ along a given amplitude $\epsilon=\epsilon_0$, we obtain the FRC associated to the forcing amplitude $\epsilon_o$. In Fig~\ref{fig:FRS-lin}, the gray curves represent the FRCs at various values of $\epsilon$. These FRCs intersect with the blue curve (ridge)  at the maxima of the response amplitude with respect to the forcing frequency $\Omega$. Hence, these local maxima on the FRCs satisfy the relationship
\begin{equation}
\label{eq:Cridge}
    \frac{\partial C(\Omega,\epsilon_o)}{\partial\Omega}=0\implies \Omega^\ast=\sqrt{1-2\zeta^2}.
\end{equation}
The ridge plotted as the blue line in Fig.~\ref{fig:FRS-lin} connects these local maxima for various forcing amplitudes and is given by the expression $C(\Omega^\ast,\epsilon)$.

The projection of the ridge onto the plane  $(\Omega,C)$ provides the damped backbone curve. As seen in the lower panel of Fig.~\ref{fig:FRS-lin}, the damped backbone curve is a straight line perpendicular to the $\Omega$ axis because the system is linear. We also observe that the damped backbone curve is shifted from the conservative backbone curve given by $\Omega=1$. This shift is controlled by the damping coefficient $\zeta$ according to the relation~\eqref{eq:Cridge}. When nonlinearity is added to the system, the backbone curves are expected to change shape. In fact, these nonlinear behaviors can be very sensitive to the damping coefficient, as observed in ref.~\cite{renson2016numerical}, where the hardening-type backbone curve associated with a conservative nonlinear system becomes a softening-type curve when sufficiently large damping is added to the system.  Thus, it is important to compute the FRS and its ridges in damped nonlinear systems for an appropriate characterization of their forced response.

\subsection{High-dimensional nonlinear systems}

Now we consider a more general setup for a nonlinear mechanical system as
\begin{equation} 
\label{eq:eom-2nd}
\boldsymbol{M}\ddot{\boldsymbol{x}}+\boldsymbol{C}\dot{\boldsymbol{x}}+\boldsymbol{K}\boldsymbol{x}+\boldsymbol{f}(\boldsymbol{x},\dot{\boldsymbol{x}})=\epsilon \boldsymbol{f}^{\mathrm{ext}}(\Omega t),\quad 0\leq\epsilon\ll1
\end{equation}
where $\boldsymbol{x}\in\mathbb{R}^n$ is the generalized displacement vector; $\boldsymbol{M},\boldsymbol{C},\boldsymbol{K}\in\mathbb{R}^{n\times n}$ are the mass, damping and stiffness matrices; $\boldsymbol{f}(\boldsymbol{x},\dot{\boldsymbol{x}})$ is a $C^r$ smooth nonlinear function that satisfies
$\boldsymbol{f}(\boldsymbol{x},\dot{\boldsymbol{x}})\sim \mathcal{O}(|\boldsymbol{x}|^2,|\boldsymbol{x}||\dot{\boldsymbol{x}}|,|\dot{\boldsymbol{x}}|^2)$; and $\epsilon \boldsymbol{f}^{\mathrm{ext}}(\Omega t)$ denotes external harmonic excitation. %Here we have $n\gg1$, namely, the mechanical system~\eqref{eq:eom-2nd} is of high dimension.

Let $\boldsymbol{z}=(\boldsymbol{x},\dot{\boldsymbol{x}})$ be the state vector of the system, the equations of motion~\eqref{eq:eom-2nd} can be transformed into a first-order system as below
\begin{equation}
\label{eq:full-first}
\boldsymbol{B}\dot{\boldsymbol{z}}	=\boldsymbol{A}\boldsymbol{z}+\boldsymbol{F}(\boldsymbol{z})+\epsilon\boldsymbol{F}^{\mathrm{ext}}({\Omega t}),
\end{equation}
where the choice for the coefficient matrices $\boldsymbol{B}$ and $\boldsymbol{A}$, and the vector-valued functions $\boldsymbol{F}$ and $\boldsymbol{F}^\mathrm{ext}$ is not unique; specific expressions can be found in~\cite{jain2022compute,part-i}. 

Similarly to the motivating example, we are interested in the periodic response of system~\eqref{eq:full-first}. In particular, we seek $\boldsymbol{z}(t)$ for $t\in[0,T]$ that satisfies the periodic boundary conditions (PBCs) 
\begin{equation}
\label{eq:pbc}
    \boldsymbol{z}(0)-\boldsymbol{z}(T)=\boldsymbol{0},
\end{equation}
where $T=2k\pi/\Omega$ is the time period of the periodic response with $k\in\mathbb{N}$ and $k>1$ defines a subharmonic response.

%\begin{sloppypar}    
In the motivating example, we have a single-degree-of-freedom oscillator and it is natural to use its amplitude to represent the response of the system. However, for the high-dimensional system~\eqref{eq:full-first}, we need an appropriate functional to quantify a response amplitude associated with the entire system. Here, we consider two such functionals that are based on two commonly used norms. Let $\mathcal{I}\subset\{1,\cdots,2n\}$ be a set of indices such that the amplitude of components $\boldsymbol{z}_\mathcal{I}$ needs to be optimized. Then, a response amplitude $\mathcal{A}_{\mathcal{L}^2}$ can be defined as
\begin{equation}
\label{eq:al2}
    \mathcal{A}_{\mathcal{L}^2}(\boldsymbol{z}(t))=\sqrt{\frac{1}{T}\int_0^T \boldsymbol{z}_\mathcal{I}^\ast(t)\boldsymbol{Q}\boldsymbol{z}_\mathcal{I}(t)\mathrm{d}t},
\end{equation}
where $\boldsymbol{Q}$ is an appropriately defined weight matrix. For instance, $\mathcal{A}_{\mathcal{L}^2}(\boldsymbol{z}(t))$ could represent a time-averaged kinetic energy of the system with appropriate choices for $\mathcal{I}$ and $\boldsymbol{Q}$. In the special case that $\mathcal{I}$ has only one element, namely $\mathcal{I}=\mathrm{opt}\in\{1,\cdots,2n\}$, we also consider the amplitude of the periodic signal for the component `opt' of the state $\boldsymbol{z}$ as
\begin{equation}
\label{eq:ainf}
    \mathcal{A}_{\mathcal{L}^\infty}(\boldsymbol{z}(t))=\max_{0\leq t\leq T} |z_\mathrm{opt}(t)|.
\end{equation}

%\textcolor{red}{I suggest simplifying the notation as $\mathcal{A}_{\mathcal{L}^2} \to \mathcal{A}_{\mathcal{I}}$ and $ \mathcal{A}_{\mathcal{L}^\infty}\to \mathcal{A}_{\mathrm{opt}}$}
%\begin{remark}
%The $\mathcal{L}^2$ norm-based amplitude is singular when it is zero, which corresponds to a zero-forcing amplitude ($\epsilon=0$). In practical computations, we set the lower bound of $\epsilon$ to be some small positive numbers. The limitation above is then not an issue as the system's response is linear in the limit $\epsilon\to0$ and we are more concerned with the response when $\epsilon$ is away from zero.
%\end{remark}

%\begin{remark}\label{rmk2}
%The $\mathcal{L}^\infty$ norm-based amplitude is non-smooth. So, the gradient of this functional is not well-defined. To resolve this issue, we consider $z_\mathrm{opt}(t)$ as an optimization objective directly and take $t\in[0, T]$ as a design variable that is to be determined, as detailed later.
%\end{remark}

An FRS is a two-dimensional surface in the space $(\mathcal{A}_{\mathcal{L}^2},\Omega,\epsilon)$ or $(\mathcal{A}_{\mathcal{L}^\infty},\Omega,\epsilon)$. Each point on the manifold is a periodic orbit of~\eqref{eq:full-first}. Since we generally do not have analytical solutions such as~\eqref{eq:C}, we need to use numerical continuation to compute the FRS via atlas algorithms. In particular, we discretize a periodic orbit using a collocation mesh and this mesh is allowed to adapt under the variations in $\epsilon$ and $\Omega$~\cite{dankowicz2013recipes}. Multidimensional manifold continuation methods for such an adaptive boundary-value problem have become available very recently~\cite{dankowicz2020multidimensional}. In principle, one can directly apply the algorithm in~\cite{dankowicz2020multidimensional} for FRS computation. However, this is computationally expensive for high-dimensional problems. 

% \textcolor{red}{what do we do instead to alleviate this computational burden?}%Interested readers can refer to~\cite{dankowicz2013recipes,dankowicz2020multidimensional} for more details on the multidimensional manifold continuation.
%As we have seen in the motivating example, the ridges (and trenches) on an FRS characterize the main features of the FRS. Therefore, one may simply locate the ridges and trenches directly without computing the FRS. 

A characterization of the FRS via its ridges and trenches is computationally efficient relative to computation of the entire FRS. We recall that the ridges and trenches on the FRS are curves of extrema of a one-parameter family of FRCs under the variation in $\epsilon$. We now define optimization problems to locate the ridges and trenches on an FRS directly. 

Based on the two types of functional characterizing the response amplitude, we consider the following two \emph{dynamic} optimization problems to locate ridges and trenches on the FRS: 
\begin{problem}\label{P1}
For any $\epsilon\in[\epsilon_\mathrm{lb},\epsilon_\mathrm{ub}]$, find $\Omega\in[\Omega_\mathrm{lb},\Omega_\mathrm{ub}]$ that renders $A_{\mathcal{L}^2}(\boldsymbol{z}(t))$ stationary under the constraints that ODEs~\eqref{eq:full-first} and PBCs~\eqref{eq:pbc} are satisfied. Here `lb' and `ub' denote lower and upper bounds that specify the domain of the response surface.
\end{problem}

\begin{problem}\label{P2}
For any $\epsilon\in[\epsilon_\mathrm{lb},\epsilon_\mathrm{ub}]$, find $t\in[0,T]$ and $\Omega\in[\Omega_\mathrm{lb},\Omega_\mathrm{ub}]$ that render $z_\mathrm{opt}(t)$ stationary under the constraints that ODEs~\eqref{eq:full-first} and PBCs~\eqref{eq:pbc} are satisfied.
\end{problem}

Here we seek \emph{stationary} solutions such that both ridges and trenches are obtained via a unified optimization problem. 

\begin{remark}\label{rmk2}
The functional $\mathcal{A}_{\mathcal{L}^\infty}$ 
 is based on the $\mathcal{L}^\infty$ norm and is not smooth. This makes gradient-based optimization difficult with the objective $\mathcal{A}_{\mathcal{L}^\infty}$. Hence, in ~\eqref{P2}, we consider $z_\mathrm{opt}(t)$ as an optimization objective and take $t\in[0, T]$ as a design variable that must be determined.
\end{remark}

%We include $t$ as a design variable in Problem~\ref{P2} to extract the amplitude defined in~\eqref{eq:ainf} (cf. Remark~\ref{rmk2}).

% We solve Problems~\ref{P1}-\ref{P2} using indirect methods

% can be solved using a parameter continuation technique.

We have now defined two \emph{dynamic} optimization problems in terms of periodic orbits of the high-dimensional system~\eqref{eq:full-first}. The computational cost of obtaining these periodic orbits is significant. Optimization of these periodic orbits further adds to the computational expense. Next, we perform model reduction of the high-dimensional system~\eqref{eq:full-first} using spectral submanifolds (SSMs),  which has two significant computational benefits in addition to the reduction in the system dimension. First, using SSM-based ROMs, the computation of periodic orbits is transformed into the computation of fixed points, which significantly speeds up the computation of the reduced periodic response. Second, it allows us to convert Problems~\ref{P1} and \ref{P2} into \emph{algebraic} optimization problems defined for the SSM-based ROM. As we will see, these reduced {algebraic} optimization problems can be solved very efficiently.

\section{SSM-based model reduction}
\subsection{Setup}
Let $\{\lambda_i\}_{i=1}^{2n}$ be a set of eigenvalues of the matrix pair $(\boldsymbol{A},\boldsymbol{B})$ in~\eqref{eq:full-first} such that $\boldsymbol{A}\boldsymbol{v}_i=\lambda_i\boldsymbol{B}\boldsymbol{v}_i$ for some nontrivial vector $\boldsymbol{v}_i$. We assume that the trivial equilibrium $\boldsymbol{z}=\boldsymbol{0}$ is asymptotically stable. Then we can arrange these eigenvalues according to $\mathrm{Re}(\lambda_{2n})\leq\mathrm{Re}(\lambda_{2n-1})\leq\cdots\leq\mathrm{Re}(\lambda_{1})<0$. We consider a $2m$-dimensional \emph{master} underdamped modal subspace 
\begin{equation}
\mathcal{E}=\Span\{\boldsymbol{v}^\mathcal{E}_1,\bar{\boldsymbol{v}}^\mathcal{E}_1,\cdots,\boldsymbol{v}^\mathcal{E}_m,\bar{\boldsymbol{v}}^\mathcal{E}_m\},
\end{equation}
where the complex eigenvalues in the spectrum of $\mathcal{E}$ are allowed satisfy a near inner (internal) resonance relationship of the form
\begin{equation}
\label{eq:res-inner}
\lambda_i^\mathcal{E}\approx\boldsymbol{l}\cdot\boldsymbol{\lambda}^\mathcal{E}+\boldsymbol{j}\cdot\bar{\boldsymbol{\lambda}}^\mathcal{E},\quad \bar{\lambda}_i^\mathcal{E}\approx\boldsymbol{j}\cdot\boldsymbol{\lambda}^\mathcal{E}+\boldsymbol{l}\cdot\bar{\boldsymbol{\lambda}}^\mathcal{E}
\end{equation}
for some $i\in\{1,\cdots,m\}$, where $\boldsymbol{l},\boldsymbol{j}\in\mathbb{N}_0^m$ (the subscript 0 here emphasizes that zero is included) satisfying $|\boldsymbol{l}+\boldsymbol{j}|:=\sum_{k=1}^m (l_k+j_k)\geq2$, and $\boldsymbol{\lambda}_\mathcal{E}=(\lambda^\mathcal{E}_1,\cdots,\lambda^\mathcal{E}_m)$. 

As an example of the inner resonance relationship ~\eqref{eq:res-inner}, we consider an internally resonant system such that the master subspace $\mathcal{E}$ has two pairs of modes that exhibit near 1:1 inner resonances, i.e., $\lambda_2^\mathcal{E}\approx\lambda_1^\mathcal{E}$ and $\bar{\lambda}_2^\mathcal{E}\approx\bar{\lambda}_1^\mathcal{E}$. Then we have
\begin{gather}
	\lambda_1^\mathcal{E}\approx l_{11}\lambda_1^\mathcal{E}+l_{12}\lambda_2^\mathcal{E}+j_{11}\bar{\lambda}_1^\mathcal{E}+j_{12}\bar{\lambda}_2^\mathcal{E},\nonumber\\
	\lambda_2^\mathcal{E}\approx l_{21}\lambda_1^\mathcal{E}+l_{22}\lambda_2^\mathcal{E}+j_{21}\bar{\lambda}_1^\mathcal{E}+j_{22}\bar{\lambda}_2^\mathcal{E}
\end{gather}
for all $l_{ik},j_{ik}\in\mathbb{N}_0$ that satisfy $l_{i1}+l_{i2}=j_{i1}+j_{i2}+1$. For systems without internal resonance, we consider single-mode SSMs with $m=1$. 

We further allow for the forcing frequency $\Omega$ to be (nearly) resonant with the master eigenvalues as~\cite{part-i}
\begin{equation}
\label{eq:res-forcing}
    \boldsymbol{\lambda}^{\mathcal{E}}-\mathrm{i}\boldsymbol{r}\Omega\approx0, \,\,\boldsymbol{r}\in\mathbb{Q}^m.
\end{equation}
To illustrate the external resonance ~\eqref{eq:res-forcing}, we again consider the example where the master subspace $\mathcal{E}$ has two pairs of modes that exhibit near 1:1 inner resonances. Provided that we are interested in the primary resonance of the first pair of modes, namely, $\Omega\approx\mathrm{i}\lambda_1^{\mathcal{E}}$, then we have $\boldsymbol{r}=(1,1)$.

\subsection{Time-periodic spectral submanifolds}
Under the addition of the nonlinear internal force $\boldsymbol{F}(\boldsymbol{z})$ and the external forcing $\epsilon\boldsymbol{F}^{\mathrm{ext}}({\Omega t})$, the master subspace $\mathcal{E}$ is perturbed into a periodic invariant manifold with period $2\pi/\Omega$ that is $\mathcal{O}(\epsilon)$ $C^r$-close to $\mathcal{E}$ near $\boldsymbol{z}=0$. There are actually many such manifolds in general, but there is a unique, smoothest one called \emph{spectral submanifold} under appropriate non-resonance conditions~\cite{haller2016nonlinear}. We denote this periodic SSM of system~\eqref{eq:full-first} by $\mathcal{W}(\mathcal{E},\Omega t)$ and review the conditions for its existence and uniqueness in Theorem~\ref{th:SSM-existence-uniqueness} of Appendix~\ref{sec:app-ssm-existence}.

\subsection{Periodic orbits as fixed points of the reduced dynamics}
As detailed in Appendix~\ref{sec:app-ssm-existence}, we have an SSM parameterization $\boldsymbol{z}=\boldsymbol{W}_\epsilon(\boldsymbol{p},\phi)$ that maps the reduced coordinates $(\boldsymbol{p},\phi)\in\mathbb{C}^{2m}\times S^1$ to the state vector of the full system. Furthermore,
the reduced dynamics $\dot{\boldsymbol{p}}=\boldsymbol{R}_\epsilon(\boldsymbol{p},\phi)$ and $\dot{\phi}=\Omega$ (see~\eqref{eq:red-dyn}) on the SSM represent a ROM for the full system~\eqref{eq:full-first}. We can simplify the reduced dynamics~\eqref{eq:red-dyn} via a normal-form style of parameterization \cite{haro2016parameterization,jain2022compute,part-i}. Here, we present the general expression for this simplified vector field. A detailed derivation can be found in~\cite{part-i}.

Let $q_i$ and $\bar{q}_i$ denote the parameterization coordinates corresponding to the modes $\boldsymbol{v}_i^{\mathcal{E}}$ and $\bar{\boldsymbol{v}}_i^{\mathcal{E}}$, respectively, then the reduced coordinates $\boldsymbol{p}$ are given as
\begin{equation}
\label{eq:p-to-qs}
\boldsymbol{p}=(q_1,\bar{q}_1,\cdots,q_m,\bar{q}_m).
\end{equation}
We rewrite the parameterization~\eqref{eq:p-to-qs} in the time-periodic polar form as
\begin{equation}
\label{eq:polar-form}
    q_i=\rho_ie^{\mathrm{i}(\theta_i+r_i\Omega t)},\,\,\bar{q}_i=\rho_ie^{-\mathrm{i}(\theta_i+r_i\Omega t)}, \quad i=1,\cdots,m.
\end{equation}
The reduced dynamics~\eqref{eq:red-dyn} on the $2m$-dimensional SSM can then be transformed into the  polar coordinates $(\rho_i,\theta_i)$ as~\cite{part-i}
\begin{gather}
\begin{pmatrix}\dot{\rho}_i\\\dot{\theta}_i\end{pmatrix}=\boldsymbol{r}^{\mathrm{p}}_i(\boldsymbol{\rho},\boldsymbol{\theta},\Omega,\epsilon)+\mathcal{O}(\epsilon|\boldsymbol{\rho}|)g_i^\mathrm{p}(\phi),\quad i=1,\cdots,m,\nonumber\\
\dot{\phi}=\Omega, \quad (\boldsymbol{\rho},\boldsymbol{\theta})\in\mathbb{R}^m\times\mathbb{T}^m.\label{eq:ode-reduced-slow-polar}
\end{gather}
Here the superscript {p} stands for `polar'; the explicit expression for $\boldsymbol{r}^{\mathrm{p}}_i$ can be found in~\cite{part-i}; $g_i^\mathrm{p}$ is a $2\pi$-periodic function.

The ROM~\eqref{eq:ode-reduced-slow-polar} becomes singular when $\rho_i\to0$~\cite{part-i}. To resolve this singularity, an explicit ROM in Cartesian coordinates has been derived in~\cite{part-i}. Specifically, we rewrite the parameterization~\eqref{eq:p-to-qs} in the form
\begin{gather}
    q_i=q_{i,\mathrm{s}}e^{\mathrm{i}r_i\Omega t}=(q_{i,\mathrm{s}}^\mathrm{R}+\mathrm{i}q_{i,\mathrm{s}}^\mathrm{I})e^{\mathrm{i}r_i\Omega t},\nonumber\\
    \label{eq:cartesian-form}
    \bar{q}_i=\bar{q}_{i,\mathrm{s}}e^{-\mathrm{i}r_i\Omega t}=(q_{i,\mathrm{s}}^\mathrm{R}-\mathrm{i}q_{i,\mathrm{s}}^\mathrm{I})e^{-\mathrm{i}r_i\Omega t},
\end{gather}
for $i=1,\cdots,m$, where $q_{i,\mathrm{s}}^\mathrm{R}=\mathrm{Re}(q_{i,\mathrm{s}})$ and $q_{i,\mathrm{s}}^\mathrm{I}=\mathrm{Im}(q_{i,\mathrm{s}})$, then the reduced dynamics~\eqref{eq:red-dyn} on the SSM in Cartesian coordinates $(\boldsymbol{q}_{\mathrm{s}}^\mathrm{R},\boldsymbol{q}_{\mathrm{s}}^\mathrm{I})\in\mathbb{R}^m\times\mathbb{R}^m$ is obtained as below~\cite{part-i}
\begin{equation}
\label{eq:ode-reduced-slow-cartesian}
\begin{pmatrix}\dot{q}_{i,\mathrm{s}}^\mathrm{R}\\\dot{q}_{i,\mathrm{s}}^\mathrm{I}\end{pmatrix}=\boldsymbol{r}^{\mathrm{c}}_i(\boldsymbol{q}_{\mathrm{s}},\Omega,\epsilon)+\mathcal{O}(\epsilon|\boldsymbol{p}|)g_i^\mathrm{c}(\phi)
\end{equation}
for $i=1,\cdots,m$. Here, the superscript {c} stands for `Cartesian'; the explicit expression for $\boldsymbol{r}^{\mathrm{c}}_i$ can be found in~\cite{part-i}; $g_i^\mathrm{c}$ is a $2\pi$-periodic function.

As shown in~\cite{part-i}, any hyperbolic fixed point of the leading-order truncation of~\eqref{eq:ode-reduced-slow-polar} or~\eqref{eq:ode-reduced-slow-cartesian} , i.e.,
\begin{equation}
\begin{pmatrix}\dot{\rho}_i\\\dot{\theta}_i\end{pmatrix}=\boldsymbol{r}^{\mathrm{p}}_i(\boldsymbol{\rho},\boldsymbol{\theta},\Omega,\epsilon),\quad i=1,\cdots, m,\label{eq:ode-reduced-slow-polar-leading}
\end{equation} or
\begin{equation}
\label{eq:ode-reduced-slow-cartesian-leading}
\begin{pmatrix}\dot{q}_{i,\mathrm{s}}^\mathrm{R}\\\dot{q}_{i,\mathrm{s}}^\mathrm{I}\end{pmatrix}=\boldsymbol{r}^{\mathrm{c}}_i(\boldsymbol{q}_{\mathrm{s}},\Omega,\epsilon),\quad i=1,\cdots,m,
\end{equation}
corresponds to a periodic solution $\boldsymbol{p}(t)$ of the reduced dynamics~\eqref{eq:red-dyn} on the SSM, $\mathcal{W}(\mathcal{E},\Omega t)$. In addition, the stability type of a hyperbolic fixed point of ~\eqref{eq:ode-reduced-slow-polar-leading} or~\eqref{eq:ode-reduced-slow-cartesian-leading} coincides with the stability type of the corresponding periodic solution on $\mathcal{W}(\mathcal{E},\Omega t)$~\cite{part-i}.

Therefore, we can obtain periodic orbits of the full, high-dimensional system~\eqref{eq:full-first} as fixed points of the low-dimensional SSM-based ROM given by~\eqref{eq:ode-reduced-slow-polar-leading} or~\eqref{eq:ode-reduced-slow-cartesian-leading}. This simplification enables us to compute the FRS via analytic prediction or multi-dimensional manifold continuation of fixed points of the ROM, as we detail in the next section. In addition, it enables us to convert the optimization problems~\ref{P1} and~\ref{P2} into algebraic optimization problems, as we will show in Sect.~\ref{sec:opt-ssm}.

% In addition, the two-dimensional manifold of the equilibria of the ROMs is mapped to the frequency response surface of the full system~\eqref{eq:full-first}.

\section{Computation of FRS via SSM-based ROMs}
\label{sec:frs-ssm}

\subsection{Simplification of response amplitude}
We now derive an explicit expression for the response amplitude of the periodic orbit of the full system~\eqref{eq:full-first}. For a fixed point of the ROM~\eqref{eq:ode-reduced-slow-polar-leading} or~\eqref{eq:ode-reduced-slow-cartesian-leading}, we obtain the corresponding periodic orbit of the full system as ~\cite{jain2022compute,part-i}
\begin{equation}
\label{eq:ssm-time-varying}
\boldsymbol{z}(t)\approx    \boldsymbol{W}(\boldsymbol{p}(t))+\epsilon\left(\boldsymbol{x}_{\boldsymbol{0}}e^{\mathrm{i}\Omega t}+\bar{\boldsymbol{x}}_{\boldsymbol{0}}e^{-\mathrm{i}\Omega t}\right),
\end{equation}
where $\boldsymbol{W}$ is a polynomial function of $\boldsymbol{p}$ and $\boldsymbol{x}_{\boldsymbol{0}}$ is the solution to a system of linear equations below~\cite{jain2022compute,part-i}
\begin{equation}
\label{eq:expphi-}
(\boldsymbol{A}-\mathrm{i}\Omega\boldsymbol{B})\boldsymbol{x}_{\boldsymbol{0}}=\boldsymbol{B}\boldsymbol{W}_{\mathbf{I}}\boldsymbol{s}_{\boldsymbol{0}}^+-\boldsymbol{F}^{\mathrm{a}}.
\end{equation}
Here, $\boldsymbol{s}_{\boldsymbol{0}}^+$ and $\boldsymbol{F}^\mathrm{a}$ are independent of $\Omega$ and their explicit expressions are available in~\cite{part-i}. We have used a leading-order truncation of the non-autonomous part of the SSM in~\eqref{eq:ssm-time-varying}, which is consistent with the truncation in the ROMs~\eqref{eq:ode-reduced-slow-polar-leading}  and~\eqref{eq:ode-reduced-slow-cartesian-leading}.

For $\epsilon\ll1$, one may simply approximate the SSM expansion~\eqref{eq:ssm-time-varying} as
\begin{equation}
\label{eq:ssm-time-independent}
\boldsymbol{z}(t)\approx\boldsymbol{W}(\boldsymbol{p}(t)),
\end{equation}
which avoids the need to solve the linear system~\eqref{eq:expphi-}. Indeed, solving system~\eqref{eq:expphi-} can be computationally expensive, especially if the full system is high-dimensional, as we must repeat this computation for every different value of $\Omega$~\cite{part-i}. The approximation~\eqref{eq:ssm-time-independent} has also been adopted in the method of normal forms~\cite{touze2006nonlinear,vizzaccaro2021direct}. In this study, we refer to the SSM solution as time-independent (TI) if~\eqref{eq:ssm-time-independent} is used and as time-varying (TV)  if~\eqref{eq:ssm-time-varying} is used. In practice, some samples of $\Omega$ near the external resonance can be taken and SSM solutions $\boldsymbol{z}(t)$ in~\eqref{eq:ssm-time-independent} and~\eqref{eq:ssm-time-varying}  can be compared to decide whether TI SSM solutions are sufficient or TV SSM solutions are required.

\begin{sloppypar}
Substituting relations~\eqref{eq:ssm-time-varying},~\eqref{eq:p-to-qs}, and~\eqref{eq:polar-form} or~\eqref{eq:cartesian-form} into \eqref{eq:al2}, the  functional $\mathcal{A}_{\mathcal{L}^2}(\boldsymbol{z}(t))$ can be simplified as
\begin{equation}
\label{eq:obj-l2}
    A_{\mathcal{L}^2}(\boldsymbol{y},\Omega,\epsilon)=\sqrt{\sum_{\hat{r}_i\in\hat{\mathcal{R}}}\hat{\boldsymbol{w}}_{\hat{r}_i,\mathcal{I}}^\ast\boldsymbol{Q}\hat{\boldsymbol{w}}_{\hat{r}_i,\mathcal{I}}}~,
\end{equation}
where ${\boldsymbol{y}}=({\rho}_1,{\theta}_1,\cdots,\rho_m,\theta_m)$ or ${\boldsymbol{y}}=({q}_{1,\mathrm{s}}^{\mathrm{R}},{q}_{1,\mathrm{s}}^{\mathrm{I}},\cdots,{q}_{m,\mathrm{s}}^{\mathrm{R}},{q}_{m,\mathrm{s}}^{\mathrm{I}})$, depending on the choice of polar or Cartesian coordinates. The detailed derivation of \eqref{eq:obj-l2} is given in Appendix~\ref{sec:appendix-th3}, where we let $\boldsymbol{W}(\boldsymbol{p})=\sum_{(\boldsymbol{c},\boldsymbol{d})}\boldsymbol{w}_{(\boldsymbol{c},\boldsymbol{d})}\boldsymbol{q}^{\boldsymbol{c}}\bar{\boldsymbol{q}}^{\boldsymbol{d}}$ and $\hat{\mathcal{R}}=\{\hat{r}:\hat{r}=(\boldsymbol{c}-\boldsymbol{d})\cdot\boldsymbol{r}\}$. 

Similarly, $z_\mathrm{opt}$ in~\eqref{eq:ainf} can be simplified as 
\begin{equation}
\label{eq:obj-opt}
    A_\mathrm{opt}(\boldsymbol{y},\Omega,\epsilon,t)=\sum_{\hat{r}_i\in\hat{\mathcal{R}}}\hat{\boldsymbol{w}}_{\hat{r}_i,\mathrm{opt}}e^{\mathrm{i}{\hat{r}_i}\Omega t},
\end{equation}
where
\begin{equation}
\label{eq:what-ri}
    \hat{\boldsymbol{w}}_{\hat{r}_i}= 
\begin{cases}
    {\boldsymbol{w}}_{\hat{r}_i}^\mathrm{cor}+\epsilon\boldsymbol{x}_{\boldsymbol{0}},& \text{if } \hat{r}_i= 1\\
    {\boldsymbol{w}}_{\hat{r}_i}^\mathrm{cor}+\epsilon\bar{\boldsymbol{x}}_{\boldsymbol{0}},& \text{if } \hat{r}_i=-1\\
    {\boldsymbol{w}}_{\hat{r}_i}^\mathrm{cor},              & \text{otherwise.}
\end{cases}
\end{equation}
Here ${\boldsymbol{w}}_{\hat{r}_i}^\mathrm{cor}\in \mathbb{C}^{2n}$ with the superscript $\mathrm{cor}\in\{\mathrm{p},\mathrm{c}\}$; $\hat{\boldsymbol{w}}_{\hat{r}_i,\mathcal{I}}\in\mathbb{C}^{|\mathcal{I}|}$ and $\hat{\boldsymbol{w}}_{\hat{r}_i,\mathrm{opt}}\in\mathbb{C}$ with subscript 'opt' referring to the corresponding entries from the vector $\hat{\boldsymbol{w}}_{\hat{r}_i}$. Depending on the choice of parameterization coordinates ($\mathrm{cor}=\mathrm{p}$ for polar coordinates~\eqref{eq:polar-form} and $\mathrm{cor}=\mathrm{c}$ for Cartesian coordinates~\eqref{eq:cartesian-form}), we have
\begin{gather} \boldsymbol{w}_{\hat{r}_i}^\mathrm{p}=\sum_{(\boldsymbol{c},\boldsymbol{d})\in\mathcal{T}_i}\boldsymbol{w}_{(\boldsymbol{c},\boldsymbol{d})}\boldsymbol{\rho}^{\boldsymbol{c}+\boldsymbol{d}}e^{\mathrm{i}(\boldsymbol{c}-\boldsymbol{d})\cdot\boldsymbol{\theta}},\nonumber\\ \boldsymbol{w}_{r_i}^\mathrm{c}=\sum_{(\boldsymbol{c},\boldsymbol{d})\in\mathcal{T}_i}\boldsymbol{w}_{(\boldsymbol{c},\boldsymbol{d})}\boldsymbol{q}_\mathrm{s}^{\boldsymbol{c}}\bar{\boldsymbol{q}}_\mathrm{s}^{\boldsymbol{d}},
\end{gather}
where $\mathcal{T}_i=\{(\boldsymbol{c},\boldsymbol{d}):\hat{r}_i=(\boldsymbol{c}-\boldsymbol{d})\cdot\boldsymbol{r}\}$.
\end{sloppypar}
The explicit expressions of the response amplitude depend on the coefficients, $\boldsymbol{w}_{(\boldsymbol{c},\boldsymbol{d})}$, of the SSM expansion $\boldsymbol{W}(\boldsymbol{p})$. Moreover, we need the expansion coefficients in the ROMs~\eqref{eq:ode-reduced-slow-polar-leading} and~\eqref{eq:ode-reduced-slow-cartesian-leading}. The automated computational procedure for obtaining these coefficients is documented in~\cite{jain2022compute} and implemented in an open-source package SSMTool~\cite{ssmtool21}. We use SSMTool to obtain these coefficients. Following~\cite{part-i,part-ii}, we determine the truncation order of $\boldsymbol{W}(\boldsymbol{p})$ based on the convergence of forced responses under increasing expansion orders.

\subsection{Analytical FRS via two-dimensional SSM-based model reduction}
\label{sec:FRS-ana}
We construct two-dimensional SSM-based ROMs for systems without internal resonances. In this case, we drop the subscript $i=1$ because $m=1$ and the reduced dynamics~\eqref{eq:ode-reduced-slow-polar-leading} becomes~\cite{jain2022compute}
\begin{gather}
    \dot{\rho}=a(\rho)+\epsilon\left(\mathrm{Re}(f)\cos\theta+\mathrm{Im}(f)\sin\theta\right),\nonumber\\
    \dot{\theta}=b(\rho)-\Omega+\frac{\epsilon}{\rho}\left(\mathrm{Im}(f)\cos\theta-\mathrm{Re}(f)\sin\theta\right)\label{eq:red-nonauto-th}
\end{gather}
where $f$ is a complex constant associated with the shape of the forcing and the master mode, and $a(\rho)$ and $b(\rho)$ are polynomial functions.
Solving for the fixed points of~\eqref{eq:red-nonauto-th} (by letting $\dot{\rho}=\dot{\theta}=0$), we obtain periodic orbits of the full system. These fixed points lie on the zero-level set of the functional~\cite{jain2022compute}
\begin{equation}
\label{eq:frs-rho-om}
    \mathcal{F}(\rho,\Omega,\epsilon): = a^2(\rho)+(b(\rho)-\Omega)^2\rho^2-{\epsilon^2}|f|^2.
\end{equation}
Thus, the response surface in reduced coordinates is a two-dimensional manifold of zeros of $\mathcal{F}(\rho,\Omega,\epsilon)$. Mapping this surface to physical coordinates via~\eqref{eq:obj-l2} or~\eqref{eq:obj-opt}, we obtain the FRS. 

%In practice, we use \texttt{isosurface} routine to extract the zero level set of $\mathcal{F}$ and then locate the FRS in reduced coordinates.
%Since the FRS is represented via a collection of vertices along with a set of faces where each face specifies the vertices that define the face, we can map the vertices in reduced coordinates to the ones in physical coordinates and then construct the FRS in physical coordinates using the same connections provided by the faces.

\subsection{Two-dimensional manifold continuation}
\label{sec:atlas-2d}
For systems with internal resonances, higher-dimensional SSM are required for model reduction. In that case, the analytic prediction of FRS via~\eqref{eq:frs-rho-om}  is not available. Hence, we use a two-dimensional manifold continuation algorithm to numerically cover the FRS. This method works alike for two-dimensional or higher-dimensional SSM-based ROMs.

Specifically, we use the Henderson algorithm~\cite{henderson2002multiple} to cover the FRS. In this algorithm, the solution manifold is characterized via a piecewise-polyhedral approximate tessellation.
Henderson's algorithm has already been implemented in the software package~\textsc{coco}. Here we use this package for the computation of the FRS. This two-dimensional manifold is represented by an atlas of charts. Each of these charts is characterized by four components~\cite{dankowicz2020multidimensional,dankowicz2013recipes}: a base point on the manifold, the tangent space of the manifold at the base point, a polygon in the tangent space that belongs to the approximate tessellation, and a radius of the circular bounding region (see Fig.~13.2 of~\cite{dankowicz2013recipes} for more details). In parameter continuation, new charts are constructed from old charts in the expansion stage and these new charts are merged with old charts in the consolidation stage. More details about the expansion and consolidation stages can be found in Chapter 13 of~\cite{dankowicz2013recipes} (see Figs.~13.2-13.3 of~\cite{dankowicz2013recipes} for schematic plots).

We use March, 2020-release of~\textsc{coco}~\cite{COCO} to perform this parameter continuation. The multidimensional continuation algorithm implemented in this release is well-documented in~\cite{dankowicz2020multidimensional}. In particular, the expansion stage is performed on the full set of problem variables, whereas the consolidation stage is conducted in a space defined by active continuation parameters. This decoupling is crucial for adaptive problems where the number and meaning of unknowns change dynamically~\cite{dankowicz2020multidimensional}. Since we perform continuation of fixed points of the SSM-based ROMs, adaptive problems are not involved. To accommodate the multidimensional continuation algorithm, however, we need to introduce some continuation parameters. In~\textsc{coco}, continuation parameters are used to track the values of the monitor functions defined along the solution manifold. Here, we define the set of active continuation parameters as $(\boldsymbol{y},\Omega,\epsilon,\mathbf{A})$, i.e., the parameterization coordinates $\boldsymbol{y}$, the excitation frequency and amplitude, and some observable $\mathbf{A}$ such as~\eqref{eq:al2} and~\eqref{eq:ainf} (we use their simplifications in~\eqref{eq:obj-l2} and~\eqref{eq:obj-opt}). In the case of Cartesian coordinates for $\boldsymbol{y}$, we also define monitor functions for their magnitudes to obtain the FRS in reduced coordinates. Since $\boldsymbol{y}$ and $\mathbf{A}$ are implicit functions of $\Omega$ and $\epsilon$, the solution manifold is indeed two-dimensional.

% We conclude this section 

\section{Optimization via SSM-based ROMs}
\label{sec:opt-ssm}

In the above section, we have successfully transformed the computation of the FRS of periodic orbits into the computation of a two-dimensional manifold of fixed points of an appropriate SSM-based ROM. However, the computational cost of such a two-dimensional manifold is still high relative to that of a one-dimensional manifold. At the same time, the ridges and trenches of an FRS are one-dimensional curves that characterize the skeleton of the FRS. Next, we provide a fast computation procedure to locate these ridges and trenches without the need to compute the entire FRS. As we will see from numerical examples in Sect.~\ref{sec:examples}, the computational time for locating these ridges and trenches is indeed much smaller than that of computing the FRS.

Recall that the ridges and trenches of the FRC can be located via formulated optimization problems~\ref{P1} and~\ref{P2}. With the simplified objective functions given in~\eqref{eq:obj-l2} and~\eqref{eq:obj-opt}, we are now ready to construct reduced optimization problems for Problems~\ref{P1} and~\ref{P2}.
Recall that ${\boldsymbol{y}}$ is a fixed point of the ROM~\eqref{eq:ode-reduced-slow-polar-leading} or~\eqref{eq:ode-reduced-slow-cartesian-leading}. Therefore, the original optimization problems~\ref{P1} and~\ref{P2} with constraints in the form of high-dimensional differential equations are reduced to optimization problems with a few \emph{algebraic} constraints. In particular, we have the following reduced algebraic optimization problem for the original dynamic optimization problem~\ref{P1}.

\begin{sloppypar}
\begin{problem}\label{P3}
Find $\epsilon\in[\epsilon_\mathrm{lb},\epsilon_\mathrm{ub}]$ and $\Omega\in[\Omega_\mathrm{lb},\Omega_\mathrm{ub}]$ that renders $A_{\mathcal{L}^2}(\boldsymbol{y},\Omega,\epsilon)$ stationary under the constraints that $\boldsymbol{h}(\boldsymbol{y},\epsilon,\Omega)=\boldsymbol{0}$ are satisfied. Here $\boldsymbol{h}:\mathbb{R}^{2m}\times\mathbb{R}\times\mathbb{R}\to\mathbb{R}^{2m}$ collects the vector field for $\boldsymbol{y}$. Specifically, in the case of the polar coordinates, we have ${\boldsymbol{y}}=({\rho}_1,{\theta}_1,\cdots,\rho_m,\theta_m)$ and $\boldsymbol{h}=(\boldsymbol{r}^{\mathrm{p}}_1,\cdots,\boldsymbol{r}^{\mathrm{p}}_m)$. For Cartesian coordinates, we have ${\boldsymbol{y}}=({q}_{1,\mathrm{s}}^{\mathrm{R}},{q}_{1,\mathrm{s}}^{\mathrm{I}},\cdots,{q}_{m,\mathrm{s}}^{\mathrm{R}},{q}_{m,\mathrm{s}}^{\mathrm{I}})$ and $\boldsymbol{h}=(\boldsymbol{r}^{\mathrm{c}}_1,\cdots,\boldsymbol{r}^{\mathrm{c}}_m)$.
\end{problem}
\end{sloppypar}

The constraint manifold defined by $\boldsymbol{h}=\boldsymbol{0}$ is two-dimensional and can be parameterized by the excitation frequency $\Omega$ and the amplitude $\epsilon$. Therefore, we have an FRS in the space $(\epsilon,\Omega,A_{\mathcal{L}^2}$). For a given forcing amplitude $\epsilon=\epsilon_o$, the response surface is reduced to an FRC. The local extrema of the FRC are stationary points of Problem~\ref{P3} restricted to $\epsilon=\epsilon_o$. As $\epsilon$ varies, these stationary points can be connected to form curves that define the ridges and trenches on the FRS. 

We note that we have included $\epsilon$ as a design variable in Problem~\ref{P3}. This augmentation enables us to use a successive parameter continuation technique~\cite{kernevez1987optimization,li2018staged,li2020optimization} to extract the ridges and trenches directly. In particular, we need to compute only one FRC with this technique. We first perform continuation with respect to $\Omega$ keeping $\epsilon$ fixed to obtain this FRC. To locate the ridges and trenches, we then perform continuation of the extrema on this FRC, allowing $\Omega$ and $\epsilon$ to change freely. This continuation is achieved through an augmented continuation problem consisting of both original constraints and adjoint equations~\cite{li2018staged,ahsan2022methods}. We provide a detailed discussion of this successive continuation method in Sect~\ref{sec:parameter-continuation}.

Similarly to Problem\ref{P1}, we obtain a reduced algebraic optimization problem for the dynamic optimization Problem~\ref{P2} as
\begin{problem}\label{P4}
Find $t\in[0,T]$, $\epsilon\in[\epsilon_\mathrm{lb},\epsilon_\mathrm{ub}]$ and $\Omega\in[\Omega_\mathrm{lb},\Omega_\mathrm{ub}]$ that render $A_\mathrm{opt}(\boldsymbol{y},\Omega,\epsilon,t)$ stationary under the constraints that $\boldsymbol{h}(\boldsymbol{y},\epsilon,\Omega)=\boldsymbol{0}$ are satisfied.
\end{problem}
Problem~\ref{P4} above seeks stationary solutions of the objective function $A_\mathrm{opt}$ in $(\Omega,\epsilon,t)$-space. We first locate the time  $t$ at which the periodic signal $A_\mathrm{opt}$ reaches its maximum magnitude (cf.~\eqref{eq:ainf}) for a given $(\Omega,\epsilon)$. Similarly to the solution procedure of Problem~\ref{P3}, we then compute an FRC by fixing $\epsilon$ and allowing $\Omega$ to change. Finally, upon locating the extrema on this FRC, we again set $\epsilon$ as a free parameter to locate the ridges and trenches on the FRS. Further details of this procedure are given in Sect.~\ref{sec:parameter-continuation}.

% We will locate the stationary solutions of the two problems using a indirect method. In other words, we solve for the first-order necessary conditions for the stationary solutions. Therefore, we need the gradients of the objective functions with respect to their arguments. In this subsection, we derive the explicit expressions of such gradients. These gradients can also be used in direct search of optimal solutions.
%\label{theo:grad-l2}

\section{Locating ridges and trenches using parameter continuation}
\label{sec:parameter-continuation}
In this section, we show how to locate the ridges and trenches of an FRS by solving the reduced optimization problems~\ref{P3} and~\ref{P4} via successive continuation~\cite{kernevez1987optimization,li2018staged,li2020optimization}. %We include a brief introduction to the solution procedure has been provided following the definition of Problem~\ref{P3} and Problem~\ref{P4}. Here we present full details about the solution procedure.

\subsection{Solution to Problem~\ref{P3}}
\label{sec:sol-p3}
\begin{sloppypar}
We introduce a Lagrangian
\begin{align}
    L_{\mathcal{L}^2}=&\mu_{A_{\mathcal{L}^2}}+\eta_{A_{\mathcal{L}^2}}(A_{\mathcal{L}^2}-\mu_{A_{\mathcal{L}^2}})+\nonumber\\&\eta_\epsilon(\epsilon-\mu_\epsilon)+\eta_\Omega(\Omega-\mu_\Omega)+\boldsymbol{\lambda}^\top\boldsymbol{h}.
\end{align}
where $(\mu_{A_{\mathcal{L}^2}},\mu_\epsilon,\mu_\Omega)$ are auxiliary parameters and $(\eta_{A_{\mathcal{L}^2}},\eta_\epsilon,\eta_\Omega,\boldsymbol{\lambda})$ are Lagrangian multipliers. Equating the derivatives of $L_{\mathcal{L}^2}$  to zero, we obtain
\begin{gather}
    A_{\mathcal{L}^2}-\mu_{A_{\mathcal{L}^2}}=0,\, \epsilon-\mu_\epsilon=0,\, \Omega-\mu_\Omega=0,\,\boldsymbol{h}=\boldsymbol{0},\label{eq:constraints}\\
    \frac{\partial A_{\mathcal{L}^2}}{\partial\boldsymbol{y}}\eta_{A_{\mathcal{L}^2}}+\left(\frac{\partial\boldsymbol{h}}{\partial\boldsymbol{y}}\right)^\top\boldsymbol{\lambda}=\boldsymbol{0},\label{eq:P3-adj-y}\\
    \frac{\partial A_{\mathcal{L}^2}}{\partial\Omega}\eta_{A_{\mathcal{L}^2}}+\eta_\Omega+\left(\frac{\partial\boldsymbol{h}}{\partial\Omega}\right)^\top\boldsymbol{\lambda}=0,\label{eq:P3-adj-om}\\ \frac{\partial A_{\mathcal{L}^2}}{\partial\epsilon}\eta_{A_{\mathcal{L}^2}}+\eta_\epsilon+\left(\frac{\partial\boldsymbol{h}}{\partial\epsilon}\right)^\top\boldsymbol{\lambda}=0\label{eq:P3-adj-eps},
\end{gather}
$1-\eta_{A_{\mathcal{L}^2}}=0$, and $\eta_\epsilon=\eta_\Omega=0$. Here,~\eqref{eq:constraints} represents constraints and~\eqref{eq:P3-adj-y}-\eqref{eq:P3-adj-eps} provides the adjoint equations. Due to the introduction of auxiliary parameters, the adjoint equations obtained are homogeneous and linear with respect to the Lagrange multipliers~\cite{li2018staged}. This enables us to construct an initial solution with trivial Lagrange multipliers. Explicit gradients of $A_{\mathcal{L}^2}$ with respect to $\boldsymbol{y}$, $\Omega$ and $\epsilon$ are given in Appendix~\ref{sec:appendix-grad}.
\end{sloppypar}

A continuation problem is constructed with the first-order necessary conditions~\eqref{eq:constraints}-\eqref{eq:P3-adj-eps}. The ridges and trenches in the FRS are obtained using \emph{ a successive continuation technique} as follows:
\begin{sloppypar}
\begin{enumerate}
    \item Detect extrema along the FRC for a given $\epsilon$. With $\eta_{A_{\mathcal{L}^2}}=\eta_\epsilon=\eta_\Omega=0$ and $\boldsymbol{\lambda}=\boldsymbol{0}$, the adjoint equations~\eqref{eq:P3-adj-y}-\eqref{eq:P3-adj-eps} are automatically satisfied and an FRC is obtained in the constraint manifold with fixed $\epsilon$ and free $\Omega\in [\Omega_\mathrm{lb},\Omega_\mathrm{ub}]$. Along the FRC, the extrema of $A_{\mathcal{L}^2}$ within the interval $[\Omega_\mathrm{lb},\Omega_\mathrm{ub}]$ are detected as fold points.
    \item Perform continuation along the secondary branches until  $\eta_{A_{\mathcal{L}^2}}=1$. The fold points above are also branch points~\cite{kernevez1987optimization,li2020optimization}. Along the secondary branch that passes through each branch point, the design variables $(\boldsymbol{y},\epsilon,\Omega)$ do not change, while the Lagrange multipliers $(\eta_{A_{\mathcal{L}^2}},\eta_\epsilon,\eta_\Omega,\boldsymbol{\lambda})$ vary linearly~\cite{li2020optimization}. Therefore, at each branch point, we switch the continuation from the primary branch to the secondary branch. We continue along the secondary branch until $\eta_{A_{\mathcal{L}^2}}=1$. We proceed to the next step once $\eta_{A_{\mathcal{L}^2}}=1$ is obtained for each of the continuation runs along the secondary branches.
    \item Release $\epsilon$ to locate ridges and trenches. We fix $\eta_{A_{\mathcal{L}^2}}=1$ and allow $\epsilon$ to vary, which yields another one-dimensional manifold that defines the ridges and trenches. This is because $\eta_{A_{\mathcal{L}^2}}=1$ and $\eta_\Omega=0$ along this curve. Furthermore, if $\eta_\epsilon=0$ at a point along the curve, then that point corresponds to a stationary solution on the FRS. We stop this continuation run once $\epsilon$ reaches the endpoints of the interval $[\epsilon_\mathrm{lb},\epsilon_\mathrm{ub}]$.
\end{enumerate}
\end{sloppypar}

From a geometric perspective, we require (for any fixed $\epsilon$)
\begin{equation}
\label{eq:geo-fold-l2}
    \frac{DA_{\mathcal{L}^2}}{D\Omega}=\frac{\partial A_{\mathcal{L}^2}}{\partial\Omega}+\left(\frac{\partial A_{\mathcal{L}^2}}{\partial\boldsymbol{y}}\right)^\top\frac{\partial\boldsymbol{y}}{\partial\Omega}=0
\end{equation}
to locate a point on the ridges/trenches of the FRS. Next, we show that~\eqref{eq:geo-fold-l2} holds during Step 3 of our procedure above. In Step 3, we have $\eta_{A_{\mathcal{L}^2}}=1$ and $\eta_\Omega=0$. Substituting these values into~\eqref{eq:P3-adj-y} and~\eqref{eq:P3-adj-om}, we obtain
\begin{equation}
\label{eq:part_AL2_adj}
    \frac{\partial A_{\mathcal{L}^2}}{\partial\boldsymbol{y}}=-\left(\frac{\partial\boldsymbol{h}}{\partial\boldsymbol{y}}\right)^\top\boldsymbol{\lambda},\, \frac{\partial A_{\mathcal{L}^2}}{\partial\Omega}=-\left(\frac{\partial\boldsymbol{h}}{\partial\Omega}\right)^\top\boldsymbol{\lambda}.
\end{equation}
Along the constraint manifold (for any fixed $\epsilon$), we have
\begin{equation}
\label{eq:Dy-Dom}
    D\boldsymbol{h}(\boldsymbol{y},\Omega)=\left(\frac{\partial\boldsymbol{h}}{\partial\boldsymbol{y}}\right)\partial\boldsymbol{y}+\left(\frac{\partial\boldsymbol{h}}{\partial\Omega}\right)\partial\Omega=\boldsymbol{0}
\end{equation}
and hence,
\begin{equation}
\label{eq:dydom}
    \frac{\partial\boldsymbol{y}}{\partial\Omega}=-\left(\frac{\partial\boldsymbol{h}}{\partial\boldsymbol{y}}\right)^{-1}\left(\frac{\partial\boldsymbol{h}}{\partial\Omega}\right).
\end{equation}
Substituting~\eqref{eq:part_AL2_adj} and~\eqref{eq:dydom} into~\eqref{eq:geo-fold-l2} yields (for any fixed $\epsilon$)
\begin{align}
    & \frac{DA_{\mathcal{L}^2}}{D\Omega} =\frac{\partial A_{\mathcal{L}^2}}{\partial\Omega}+\left(\frac{\partial A_{\mathcal{L}^2}}{\partial\boldsymbol{y}}\right)^\top\frac{\partial\boldsymbol{y}}{\partial\Omega}\nonumber\\
    & =-\left(\frac{\partial\boldsymbol{h}}{\partial\Omega}\right)^\top\boldsymbol{\lambda}+\boldsymbol{\lambda}^\top\left(\frac{\partial\boldsymbol{h}}{\partial\boldsymbol{y}}\right)\left(\frac{\partial\boldsymbol{h}}{\partial\boldsymbol{y}}\right)^{-1}\left(\frac{\partial\boldsymbol{h}}{\partial\Omega}\right)\equiv0.\label{eq:DAl2-Dom}
\end{align}
Thus,~\eqref{eq:geo-fold-l2} indeed holds, and the solutions from Step 3 of our procedure lie on the ridges or trenches of the FRS.

\subsection{Solution to Problem~\ref{P4}}
\label{sec:sol-p4}
We introduce a Lagrangian
\begin{align}
    L_\mathrm{opt}=\mu_{A_\mathrm{opt}}+\eta_{A_\mathrm{opt}}(A_\mathrm{opt}-\mu_{A_\mathrm{opt}})+\eta_t(t-\mu_t)\nonumber\\+\eta_\epsilon(\epsilon-\mu_\epsilon)+\eta_\Omega(\Omega-\mu_\Omega)+\boldsymbol{\lambda}^\top\boldsymbol{h}.
\end{align}
where $(\mu_{A_\mathrm{opt}},\mu_t,\mu_\epsilon,\mu_\Omega)$ are auxiliary parameters and $(\eta_{A_\mathrm{opt}},\eta_t,\eta_\epsilon,\eta_\Omega,\boldsymbol{\lambda})$ are Lagrangian multipliers. Equating the derivatives of $L_\mathrm{opt}$ to zero yields
\begin{gather}
    A_\mathrm{opt}-\mu_{A_\mathrm{opt}}=0, \quad t-\mu_t=0,\quad \epsilon-\mu_\epsilon=0,\label{eq:constraints1-inf}\\ \Omega-\mu_\Omega=0,\quad \boldsymbol{h}=\boldsymbol{0},\label{eq:constraints2-inf}\\
    \frac{\partial A_\mathrm{opt}}{\partial\boldsymbol{y}}\eta_{A_\mathrm{opt}}+\left(\frac{\partial\boldsymbol{h}}{\partial\boldsymbol{y}}\right)^\top\boldsymbol{\lambda}=\boldsymbol{0},\label{eq:P4-adj-y}\\ \frac{\partial A_\mathrm{opt}}{\partial t}\eta_{A_\mathrm{opt}}+\eta_t=0,\label{eq:P4-adj-t}\\
    \frac{\partial A_\mathrm{opt}}{\partial\Omega}\eta_{A_\mathrm{opt}}+\eta_\Omega+\left(\frac{\partial\boldsymbol{h}}{\partial\Omega}\right)^\top\boldsymbol{\lambda}=0,\\\frac{\partial A_\mathrm{opt}}{\partial\epsilon}\eta_{A_\mathrm{opt}}+\eta_\epsilon+\left(\frac{\partial\boldsymbol{h}}{\partial\epsilon}\right)^\top\boldsymbol{\lambda}=0 \label{eq:P4-adj-om-eps},
\end{gather}
$1-\eta_{A_\mathrm{opt}}=0$, and $\eta_t=\eta_\epsilon=\eta_\Omega=0$. Here,~\eqref{eq:constraints1-inf}-\eqref{eq:constraints2-inf} represent the original constraints and~\eqref{eq:P4-adj-y}-\eqref{eq:P4-adj-om-eps} provide the adjoint equations. Similarly to the previous case, we have introduced auxiliary parameters such that the adjoint equations are homogeneous and linear with respect to the Lagrange multipliers. Explicit gradients of $A_\mathrm{opt}$ with respect to $\boldsymbol{y}$, $\Omega$ and $\epsilon$ are given in Appendix~\ref{sec:appendix-grad}.

A continuation problem is constructed with the first-order necessary conditions~\eqref{eq:constraints1-inf}-\eqref{eq:P4-adj-om-eps}. Likewise, the ridges and trenches on the FRS are obtained using a successive continuation technique as below
\begin{sloppypar}
\begin{enumerate}
    \item Find the amplitude of a periodic orbit for given $\epsilon$ and $\Omega$. With $\eta_{A_\mathrm{opt}}=\eta_t=\eta_\epsilon=\eta_\Omega=0$ and $\boldsymbol{\lambda}=\boldsymbol{0}$, the adjoint equations~\eqref{eq:P4-adj-y}-\eqref{eq:P4-adj-om-eps} are automatically satisfied, and a periodic orbit is obtained for the given $\epsilon$ and $\Omega$, along which extrema of $A_\mathrm{opt}$ for $t\in[0,T]$ are detected as fold points. We identify the fold point with maximum magnitude, which gives the amplitude of the periodic orbit.
    \item Perform continuation along the secondary branch until $\eta_{A_\mathrm{opt}}=1$. The fold point above is also a branch point~\cite{kernevez1987optimization,li2020optimization}. Along the secondary branch passing through the branch point, the design variables $(\boldsymbol{y},\epsilon,\Omega)$ do not change whereas the Lagrange multipliers $(\eta_{A_\mathrm{opt}},\eta_\epsilon,\eta_\Omega,\boldsymbol{\lambda})$ vary linearly~\cite{li2020optimization}. We perform branch switching and continue along the secondary branch until $\eta_{A_\mathrm{opt}}=1$.
    \item Release $\Omega$ and perform continuation until $\eta_\Omega=0$. Once $\eta_{A_\mathrm{opt}}=1$ is obtained, we fix $\eta_{A_\mathrm{opt}}=1$ and allow $\Omega$ to vary within $[\Omega_\mathrm{lb},\Omega_\mathrm{ub}]$ to produce an FRC. This is because $\eta_{A_\mathrm{opt}}=1$ and $\eta_t=0$ along this curve. If $\eta_\Omega=0$ is detected at a point along the curve, this point corresponds to a stationary solution on the FRC. Along the FRC, several stationary points may be detected. We move on to Step 4 for each of these stationary points.
    \item Release $\epsilon$ to locate ridges and trenches. Starting from a stationary point on the FRC obtained in Step 3, we fix $\eta_\Omega=0$ but allow $\epsilon$ to vary to obtain a one-dimensional manifold defining the ridges and trenches of the FRS. This is because $\eta_{A_\mathrm{opt}}=1$ and $\eta_t=\eta_\Omega=0$ along this one-dimensional manifold. We terminate this continuation run once $\epsilon$ reaches the endpoints of the interval $[\epsilon_\mathrm{lb},\epsilon_\mathrm{ub}]$.
\end{enumerate}
\end{sloppypar}

From a geometric point of view, require (for any fixed $\epsilon$)
\begin{equation}
\label{eq:geo-aopt}
    \frac{DA_{\mathrm{opt}}}{D\Omega}:=\frac{\partial A_{\mathrm{opt}}}{\partial\Omega}+\left(\frac{\partial A_{\mathrm{opt}}}{\partial\boldsymbol{y}}\right)^\top\frac{\partial\boldsymbol{y}}{\partial\Omega}=0,\, \frac{\partial A_{\mathrm{opt}}}{\partial t}=0
\end{equation}
to locate a point on the ridges/trenches of the FRS. We need to show that~\eqref{eq:geo-aopt} holds in Step 4 of the above procedure. In this step, we have $\eta_{A_{\mathrm{opt}}}=1$ and $\eta_t=\eta_\Omega=0$. Substituting $\eta_{A_{\mathrm{opt}}}=1$ and $\eta_t=0$ into~\eqref{eq:P4-adj-t} yields $\partial A_\mathrm{opt}/\partial t=0$. Next, we show that ${DA_{\mathrm{opt}}}/{D\Omega}=0$  in Step 4 of the above procedure. We substitute $\eta_{A_{\mathrm{opt}}}=1$ and $\eta_\Omega=0$ into \eqref{eq:P4-adj-y} and~\eqref{eq:P4-adj-om-eps}, yielding
\begin{equation}
    \frac{\partial A_{\mathrm{opt}}}{\partial\boldsymbol{y}}=-\left(\frac{\partial\boldsymbol{h}}{\partial\boldsymbol{y}}\right)^\top\boldsymbol{\lambda},\, \frac{\partial A_{\mathrm{opt}}}{\partial\Omega}=-\left(\frac{\partial\boldsymbol{h}}{\partial\Omega}\right)^\top\boldsymbol{\lambda}.
\end{equation}
Along the constraint manifold (for any fixed $\epsilon$),~\eqref{eq:Dy-Dom} still holds. Then we can easily show ${DA_{\mathrm{opt}}}/{D\Omega}=0$ similarly to~\eqref{eq:DAl2-Dom}. Thus,~\eqref{eq:geo-aopt} indeed holds, and the solutions from Step 4 of our procedure lie on ridges or trenches of the FRS.

\begin{remark}
\label{rmk-sn}
In numerical experiments, we observed that when the continuation run in Step 3 approaches a saddle-node (SN) bifurcation point on the FRC, the design variables barely change while the Lagrange multiplier changes considerably. At an SN point, we have $\eta_{A_\mathrm{opt}}=1$ and $\partial\boldsymbol{h}/\partial\boldsymbol{y}$ is singular. We deduce from~\eqref{eq:P4-adj-y} that $|\boldsymbol{\lambda}|\to\infty$ when $\partial\boldsymbol{h}/\partial\boldsymbol{y}$ becomes singular and $\eta_{A_\mathrm{opt}}\neq0$. This explains the observation because $\eta_{A_\mathrm{opt}}\equiv1$ along the FRC. Therefore, we need to select initial points such that fold points can be found along a segment of the FRC that does not contain any SN points. This is feasible since the fold points do not exactly coincide with the SN points.
\end{remark}

%\begin{remark}
%In numerical experiments, we also observe that the Jacobian of Newton iteration becomes nearly singular when $\epsilon\to0$ in the continuation run in Step 4. This indicates that the trivial solution is a singular point. Indeed, when $\epsilon=0$, we have $z(t)\equiv0$, independently of $t$, and hence locating a stationary solution with respect to $t$ is not meaningful. In practice, we again set the lower bound of $\epsilon$ to be some small positive numbers.
%\end{remark}

\section{Examples}
\label{sec:examples}
\subsection{A cantilever beam with nonlinear support}
We consider a cantilever beam with a cubic spring and a cubic damper support at its free end~\cite{ponsioen2019analytic}. The beam is modeled using Bernoulli beam theory, and hence the only nonlinearity in this example comes from the support spring and damper. The reaction force of the support is modeled as
\begin{equation}
\label{eq:f-spring-damp}
F=\kappa w^3+\gamma \dot{w}^3,
\end{equation}
where $w$ is the transverse displacement at the free end, $F$ is the reaction force, and $\kappa$ and $\gamma$ denote the coefficients of the cubic spring and damper. We apply a harmonic excitation $\epsilon\cos\Omega t$ at the free end and calculate the FRS. Here, we use the amplitude functional
\begin{equation}
\label{eq:w-l2}
    ||w||_{\mathcal{L}^2}=\sqrt{\frac{1}{T}\int_0^T w^2(t)\mathrm{d}t}
\end{equation}
to characterize the response of the beam. Note that $||w||_{\mathcal{L}^2}$ is a special case of the $\mathcal{L}^2$ norm-based objective~\eqref{eq:al2}.

The geometric and material properties of this beam are the same as those in~\cite{ponsioen2019analytic}. In particular, the width, height, and length of the beam are 10 mm, 10 mm, and 2700 mm, respectively, and the density and Young's modulus of the material are $1780\times10^{-9}\,\mathrm{kg/{mm}^3}$ and $45\times10^6\,\mathrm{kPa}$. We use a classic finite element scheme to discretize this beam. Specifically, two degrees of freedom, namely, the transverse displacement and the rotation angle, are introduced at each node. At each element, the displacement field is approximated with Hermite interpolation. The resulting equations of motion for the discretized beam model are given as
\begin{equation}
\label{eq:eom-beams}
\boldsymbol{M}\ddot{\boldsymbol{x}}+\boldsymbol{C}\dot{\boldsymbol{x}}+\boldsymbol{K}\boldsymbol{x}+\boldsymbol{N}(\boldsymbol{x},\dot{\boldsymbol{x}})=\epsilon\boldsymbol{f}\cos\Omega t,
\end{equation}
where $\boldsymbol{x}\in\mathbb{R}^{2N_{\mathrm{e}}}$ is the assembly of all degrees of freedom with $N_{\mathrm{e}}$ being the number of elements used in the discretization, $\boldsymbol{M}$, $\boldsymbol{C}$ and $\boldsymbol{K}$ are the mass, damping and stiffness matrices, respectively, $\boldsymbol{N}$ collects the nonlinear internal force vector due to~\eqref{eq:f-spring-damp}, and $\boldsymbol{f}$ denotes the external force vector associated with the harmonic excitation. Here, we use the Rayleigh damping hypothesis $\boldsymbol{C}=\alpha\boldsymbol{M}+\beta\boldsymbol{K}$ with $\alpha=1.25\times10^{-4}\,\mathrm{s}$ and $\beta=2.5\times10^{-5}\,\mathrm{s}^{-1}$~\cite{ponsioen2019analytic}. The coefficients of cubic spring and damper are chosen as $\kappa=6\,\mathrm{mN/{mm}^3}$ and $\gamma=-0.02\mathrm{mNs/{mm}^3}$.

In the following computations, the beam is uniformly discretized with 25 elements ($N_\mathrm{e}=25$) and hence the system has 50 degrees of freedom and a 100-dimensional phase space. In this case, the eigenvalues corresponding to the slowest eigenspace are
\begin{equation}
    \lambda_{1,2}=-0.0062\pm7.0005\mathrm{i}.
\end{equation}
We take the slowest eigenspace as the master subspace $\mathcal{E}$, construct two-dimensional SSM-based ROMs, and use them to extract ridges and trenches on the forced response surface of this system. Here, we set $\epsilon_\mathrm{lb}=1\times10^{-4}$ and $\epsilon_\mathrm{ub}=0.01$, $\Omega_\mathrm{lb}=6.96$ and $\Omega_\mathrm{ub}=7.04$ to limit the domain of the FRS.

We first set $\epsilon=\epsilon_\mathrm{ub}$ and compute the FRC of the system based on TV-SSM solution~\eqref{eq:ssm-time-varying} and TI-SSM solution~\eqref{eq:ssm-time-independent} to check whether TI-SSM is accurate enough. We find that $\mathcal{O}(5)$ expansion is sufficient to approximate the converged FRC.
As seen in Fig.~\ref{fig:cant_beam_ti_vs_tv}, the FRC based on the TI-SSM solution matches that of the TV-SSM solution. We conclude that the TI-SSM based predictions have sufficient accuracy.  Hence, we will use the TI-SSM solution~\eqref{eq:ssm-time-independent} in the rest of this example.

\begin{figure}[!ht]
\centering
\includegraphics[width=.45\textwidth]{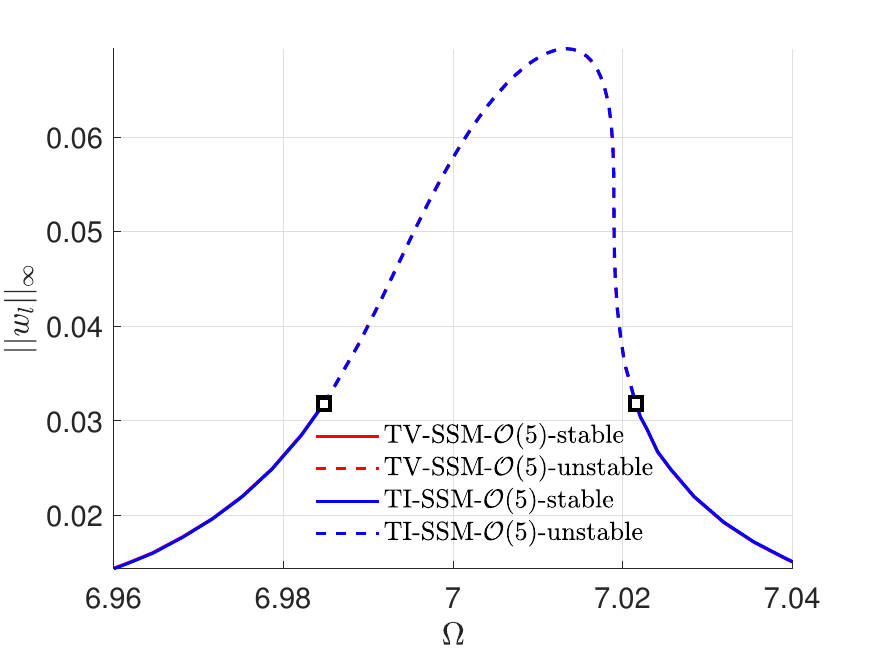}
\caption{FRC for the amplitude of periodic orbits at the end of the cantilever beam with $\epsilon=\epsilon_\mathrm{ub}=0.01$. Here and throughout this paper, $\mathcal{O}(k)$ denotes that the expansion truncation order for $\boldsymbol{W}(\boldsymbol{p})$ in~\eqref{eq:ssm-time-varying} and~\eqref{eq:ssm-time-independent}, namely, the autonomous part of SSM, is equal to $k$. TV-SSM and TI-SSM correspond to~\eqref{eq:ssm-time-varying} and~\eqref{eq:ssm-time-independent}, respectively. The blue lines coincide well with the red lines such that the red lines are nearly invisible.}
\label{fig:cant_beam_ti_vs_tv}
\end{figure}

Now we compute the FRS of the system using both the analytic prediction in Sect.~\ref{sec:FRS-ana} and the multidimensional atlas algorithm in Sect.~\ref{sec:atlas-2d}. The FRS obtained is shown in Fig.~\ref{fig:cant-frs}. In the upper panel, we show the FRS obtained from the analytic prediction, which took 7 seconds of computational time. In the lower panel, the FRS obtained from the multidimensional continuation algorithm shown, where the surface is approximated with 1500 polygons. The computational time for obtaining this FRS is about half an hour. We note that isolas (cf. the tip pointed by the green arrow in the left panel) are uncovered automatically via the FRS. By comparing the two panels, we observe that polygons near the isola region have much smaller sizes relative to other polygons away from the region. This non-uniform mesh is a result of the adaptation of continuation step sizes in the atlas algorithm. In contrast, we have used uniformly distributed grids to generate the analytic FRS. To capture the intricate surface around the tip, we used a fine mesh (with 42538 faces) for the analytic FRS.

\begin{figure}[!ht]
\centering
\includegraphics[width=.45\textwidth]{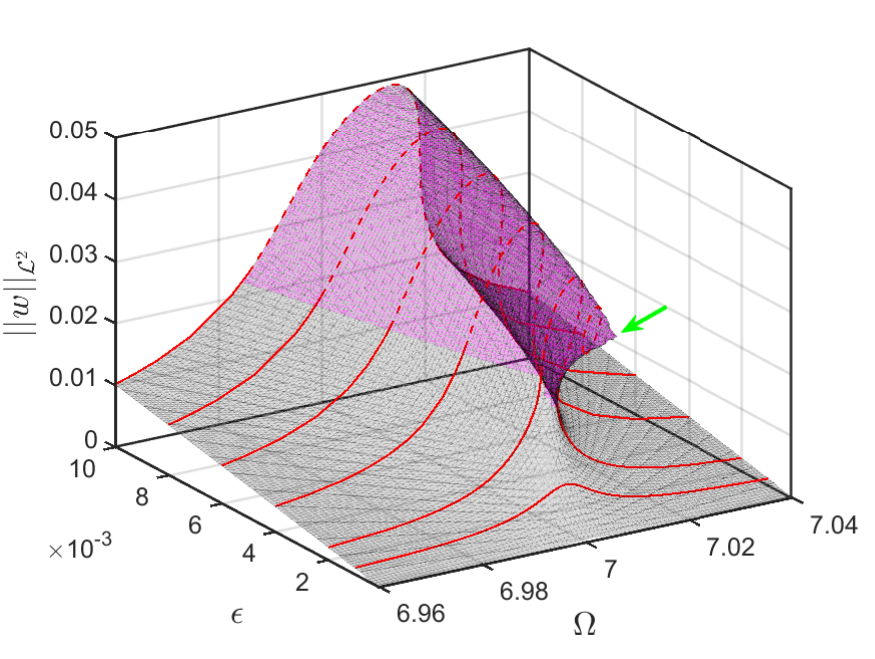}
\includegraphics[width=.45\textwidth]{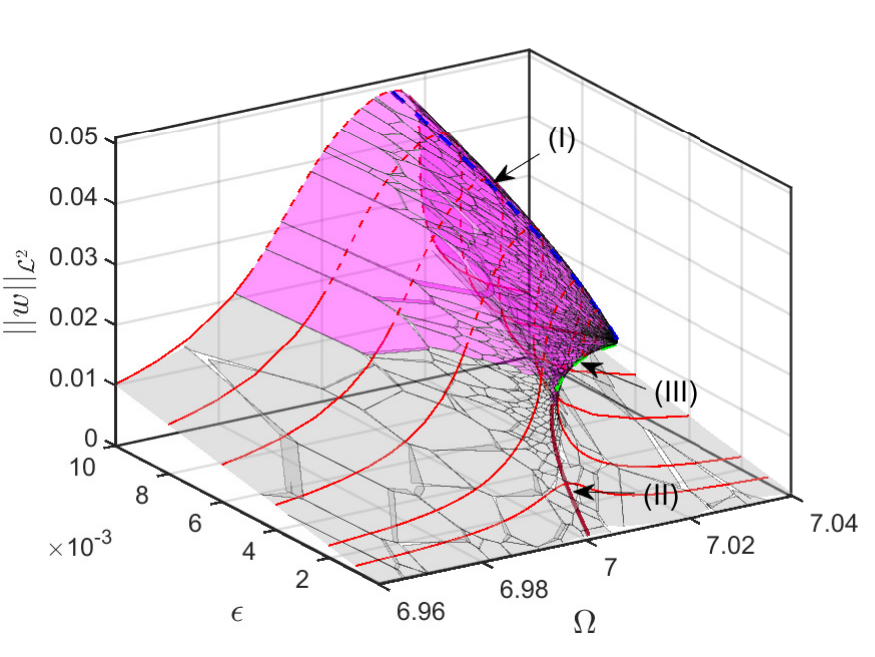}
\caption{Forced response surface of the cantilever beam with nonlinear support, obtained from the analytic prediction (upper panel), and multidimensional continuation of fixed points (lower panel). Some sampled FRCs of the full system are also presented for the purpose of validation. In the right panel, ridges, and trenches on the FRS are also provided. Here and throughout this paper, gray and magenta areas indicate stable and unstable forced responses and solid and dashed lines represent stable and unstable periodic solutions.}
\label{fig:cant-frs}
\end{figure}

To validate the above FRS obtained from SSM-based ROMS, we compute the FRCs of the full system sampled at $\epsilon\in\{1\times10^{-3},2\times10^{-3},4\times10^{-3},6\times10^{-3},8\times10^{-3},1\times10^{-2}\}$ using the collocation method implemented in the po-toolbox of \textsc{coco}. We choose the direct computation of FRCs instead of the FRS for validation because the computational cost of the FRS for such a high-dimensional system is significant. As seen in Fig.~\ref{fig:cant-frs}, the six sampled FRCs agree with the FRS obtained from SSM-based reduction. We note that the computational time for the sampled FRCs is about 6 hours, which shows a significant speed-up using the SSM-based reduction.

Next, we compute the ridges and trenches of the FRS using the solution procedure established in~\ref{sec:sol-p3}. As we will see, the computational time for locating these ridges and trenches via the successive continuation is just 31 seconds, much less than the half hour required for generating the entire FRS via the two-dimensional manifold continuation. 

We initialize $\epsilon=\epsilon_\mathrm{ub}$ and apply the successive continuation scheme to locate the ridge on the forced response curve. Indeed, we have a maximum point on the FRC shown in Fig.~\ref{fig:cant_beam_ti_vs_tv}. We expect a ridge consisting of this maximum under the variation of $\epsilon$. The maximum point is detected as a branch point and denoted by a blue cross marker on the red curve shown in Fig.~\ref{fig:cant_beam_ridges}. We follow the procedure in~\ref{sec:sol-p3} and obtain the ridge as expected. This ridge is plotted in a blue dashed line and marked as (I) in Fig.~\ref{fig:cant_beam_ridges}. All periodic orbits are unstable along this ridge, which is depicted by our use of the dashed line.

\begin{figure*}[!ht]
\centering
\includegraphics[width=.45\textwidth]{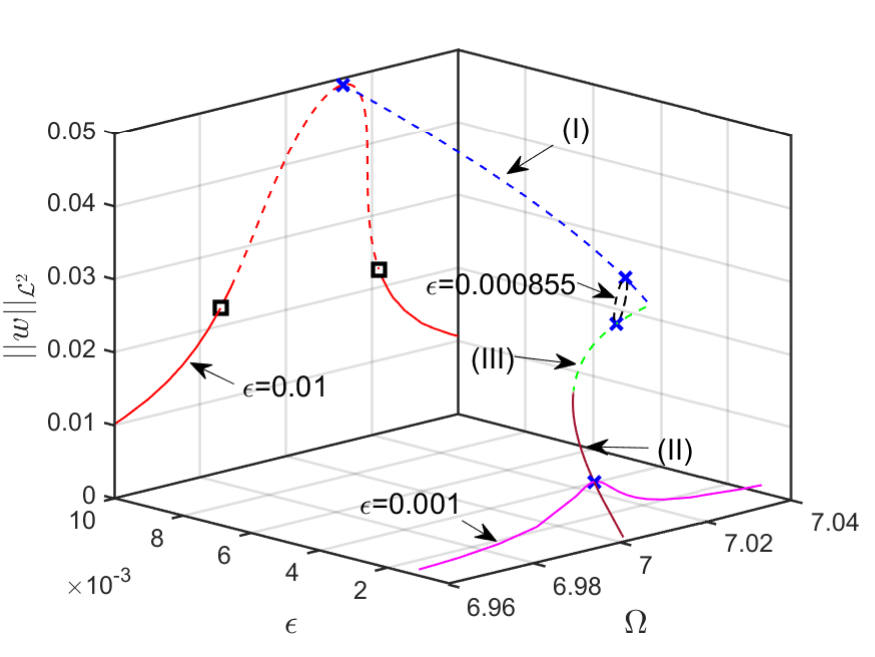}\\
\includegraphics[width=.45\textwidth]{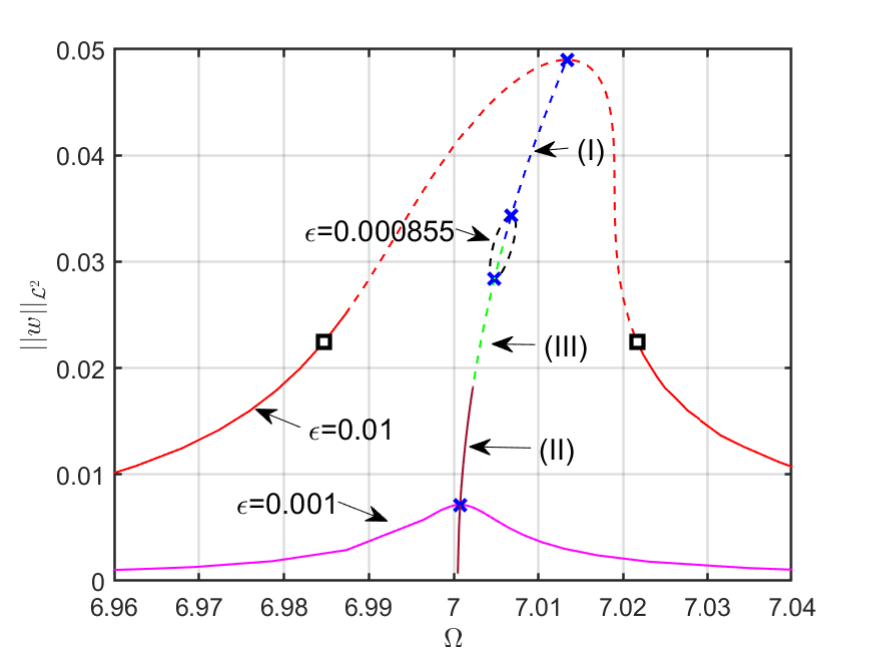}
\includegraphics[width=.45\textwidth]{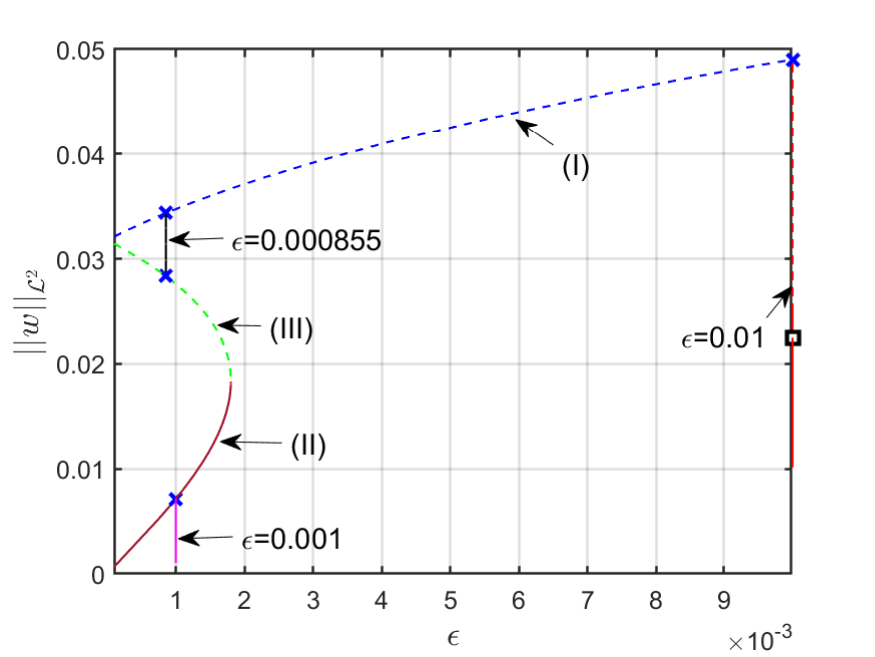}
\caption{Frequency response curves and resulting ridges and trenches in the forced response surface of the cantilever beam with cubic spring and damper support. The lower two panels give the projection of the upper panel onto $(\Omega,||w||_{\mathcal{L}^2}$ and $(\epsilon,||w||_{\mathcal{L}^2}$).}
\label{fig:cant_beam_ridges}
\end{figure*}

We note that the response along the ridge does not decrease to zero when $\epsilon\to0$. This indicates that the system admits a limit cycle in an unforced case and undergoes an isola bifurcation when $\epsilon\to0$, as illustrated in Fig.~\ref{fig:isola-simp}. Consequently, one can use the obtained ridges and trenches on the forced response surface to infer isola bifurcations.

\begin{figure}[!ht]
\centering
\includegraphics[width=.25\textwidth]{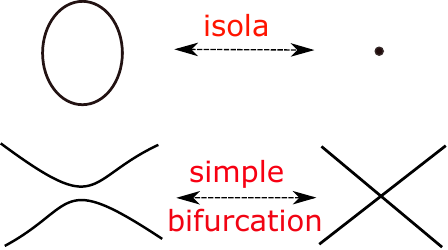}
\caption{A schematic plot of isola and simple bifurcations.}
\label{fig:isola-simp}
\end{figure}

We expect that the system has another family of periodic orbits under the addition of harmonic excitation. These periodic orbits are perturbed from the origin, which is a hyperbolic fixed point, and the response amplitude of these periodic orbits will decrease to zero when $\epsilon\to0$. We initialize with $\epsilon=0.001$ and $\Omega=6.96$ and apply the solution procedure established in~\ref{sec:sol-p3} to locate the ridge on the FRS that corresponds to this family of periodic orbits. As seen in Fig.~\ref{fig:cant_beam_ridges}, the FRC for $\epsilon=0.001$ consists of a maximum, which lies on a ridge shown as the brown solid curve marked as (II). Indeed, the response along this ridge decreases to zero when $\epsilon\to0$, as seen in Fig.~\ref{fig:cant_beam_ridges}.

In the last continuation run to generate segment (II) of the ridge, we found that when $\epsilon$ increases to a critical value of $\epsilon_\mathrm{simp}\approx 1.8020\times10^{-3}$, it cannot be increased further, as shown in the lower-right panel of Fig.~\ref{fig:cant_beam_ridges}. This critical value corresponds to a simple bifurcation illustrated in Fig.~\ref{fig:isola-simp}. At the simple bifurcation, the primary branch of the FRC is merged into the isolated (detached) branch. Consequently, the maximum on the primary branch merges with the minimum on the isolated branch when $\epsilon\to\epsilon_\mathrm{simp}$. Moreover, these two extrema disappear when $\epsilon>\epsilon_\mathrm{simp}$. This explains why $\epsilon$ cannot be increased further than $\epsilon_\mathrm{simp}$ along segment (II). Therefore, we can also use the obtained ridges and trenches to infer simple bifurcations.

As seen in Fig.~\ref{fig:isola-simp}, there is a local minimum in the isolated branch of an FRC. We initialize with a point on the segment (I) ($\epsilon=0.000855$) and apply the solution procedure established in~\ref{sec:sol-p3} to locate the trench on the FRS corresponds to this family of local minima. This trench is plotted in a green dashed line in Fig.~\ref{fig:cant_beam_ridges} and marked with (III). We see that segments (I) and (III) intersect at $\epsilon=0$ where the isola bifurcation occurs. In addition, segments (III) and (II) merge smoothly at $\epsilon=\epsilon_\mathrm{cusp}$, where the simple bifurcation is observed, as seen in the lower-right panel of Fig.~\ref{fig:cant_beam_ridges}.

We now provide a validation of the ridges and trenches obtained via our SSM-based ROMs. We apply the collocation method implemented in \textsc{coco}~\cite{COCO,ahsan2022methods,dankowicz2013recipes} to solve for Problem~\ref{P1}. In particular, the \texttt{po}-toolbox supports an automated construction of adjoint equations of periodic orbits~\cite{li2018staged}. We follow a successive continuation method similar to the procedure in~\ref{sec:sol-p3} to locate the ridges and trenches on the FRS of the full system. As seen in the left panel of Fig.~\ref{fig:cant-ridge}, the results from the collocation methods agree with the predictions from our SSM-based ROM. Here, the computation time of the collocation method on the full systems is nearly 1.5 days, while that of the SSM-based prediction is just about 31 seconds. %Therefore, the proposed model reduction has a significant speed-up gain relative to the collocation methods applied to the full system.

\begin{figure}[!ht]
\centering
\includegraphics[width=.45\textwidth]{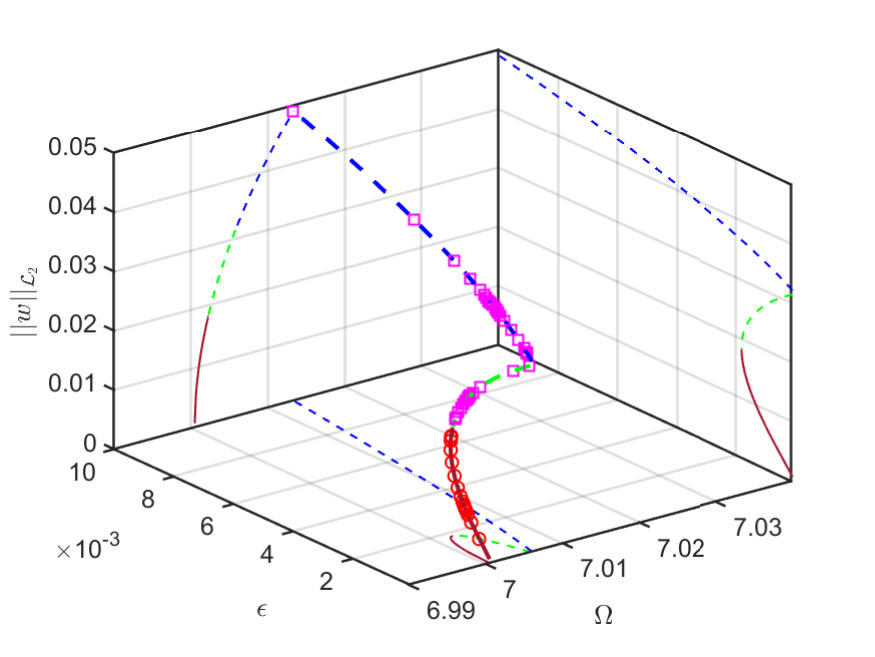}
\caption{Ridges and trenches obtained from SSM-based predictions (lines) and collocation methods (markers) applied to the full system. Here the magenta squares and red circles denote unstable and stable periodic orbits, respectively. The projection of these curves onto the three coordinate planes is also shown here.}
\label{fig:cant-ridge}
\end{figure}

\subsection{A square plate with 1:1 internal resonance}

As our second example, we consider a simply supported plate shown in the left panel of Fig.~\ref{fig:plate_mesh}~\cite{part-i,part-ii}. Here, $a$ and $b$ give the length and width of the plate. We are interested in the case of a square plate, i.e., $a=b$, such that the second and third bending modes of the system satisfy a 1:1 internal resonance because of the plate's geometric symmetry.

\begin{figure}[!ht]
\centering
\includegraphics[width=0.4\textwidth]{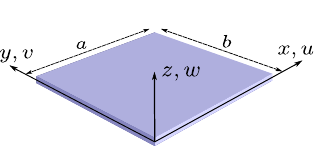}
\includegraphics[width=0.4\textwidth]{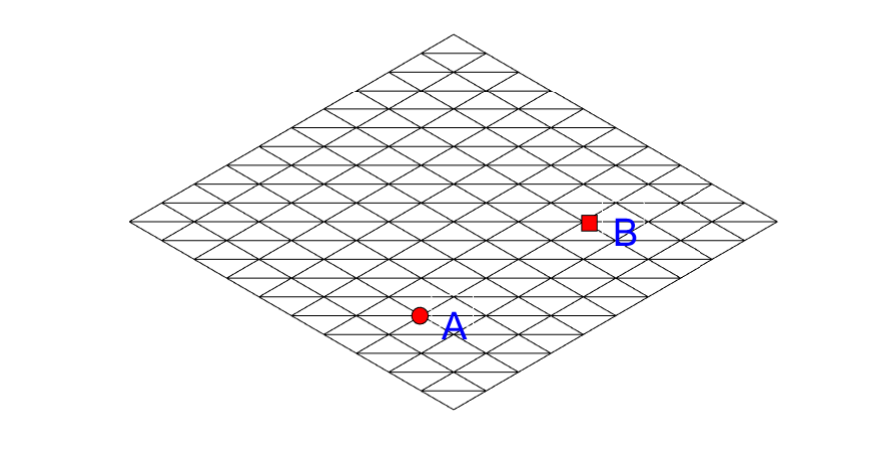}
\caption{A simply supported rectangular plate and a mesh for a square plate ($a=b$)~\cite{part-i}.}
\label{fig:plate_mesh}
\end{figure}

We model this square plate using von K\'arm\'an theory and hence the system has distributed nonlinearity, which is different from the previous example. In particular, both the in-plane and out-of-plane displacements are considered as unknowns, and nonlinear stretching forces due to large transverse displacement are taken into account. The nonlinear equations governing the motion of the plate can be found in~\cite{reddy2015introduction}.

We use flat facet shell elements developed in~\cite{allman1976simple,allman1996implementation} to discretize the unknown displacement field of this plate. Specifically, we use triangular elements to discretize the domain. An illustration of the mesh generated using the triangular elements is shown in the right panel of Fig.~\ref{fig:plate_mesh}. Here, we have 200 elements. Six degrees-of-freedom (DOFs) are introduced at each node of an element. With the boundary conditions from the simple supports applied, this discrete model has 606 DOFs, resulting in a 1212-dimensional phase space. The discrete model obtained is of the same form as~\eqref{eq:eom-beams}. We use an open-source finite element package~\cite{FEcode} to obtain the mass and stiffness matrices and the nonlinear internal force vector. For the damping matrix, we again consider the Rayleigh damping hypothesis $\boldsymbol{C}=\alpha\boldsymbol{M}+\beta\boldsymbol{K}$ with $\alpha=1$ and $\beta=4\times10^{-6}$.

We apply a transverse excitation $\epsilon100\cos\Omega t$ at point A (cf.~the lower panel of Fig.~\ref{fig:plate_mesh}) and study the forced response of the system under variations in $\epsilon$ and $\Omega$. We use the same geometric and material parameters as those in~\cite{part-i}. The natural frequencies of the second and third bending modes of the finite element model are obtained as $\omega_2 \approx 763.6\,\mathrm{rad/s}$ and $\omega_3\approx767.7\,\mathrm{rad/s}$. Their vibration mode shapes are shown in Fig.~\ref{fig:plate_mode}, where we see that points A and B represent the response of the third and second bending modes (cf.~the lower panel of Fig.~\ref{fig:plate_mesh}). We consider three measures to characterize the response of the system:
\begin{gather}
        ||w_\mathrm{A}||_{\mathcal{L}^2}=\sqrt{\frac{1}{T}\int_0^T w_\mathrm{A}^2(t)\mathrm{d}t},\nonumber\\||w_\mathrm{B}||_{\mathcal{L}^2}=\sqrt{\frac{1}{T}\int_0^T w_\mathrm{B}^2(t)\mathrm{d}t},\nonumber\\||E_\mathrm{k}||_{\mathcal{L}^2}=\sqrt{\frac{1}{2T}\int_0^T\dot{\boldsymbol{x}}^\mathrm{T}\boldsymbol{M}\dot{\boldsymbol{x}}\mathrm{d}t}\label{eq:obj-plate},
\end{gather}
where $w_\mathrm{A}$ and ${w}_\mathrm{B}$ denote the transverse displacements at points A and B, and $E_\mathrm{k}$ gives the kinetic energy of the plate.

\begin{figure}[!ht]
\centering
\includegraphics[width=0.45\textwidth]{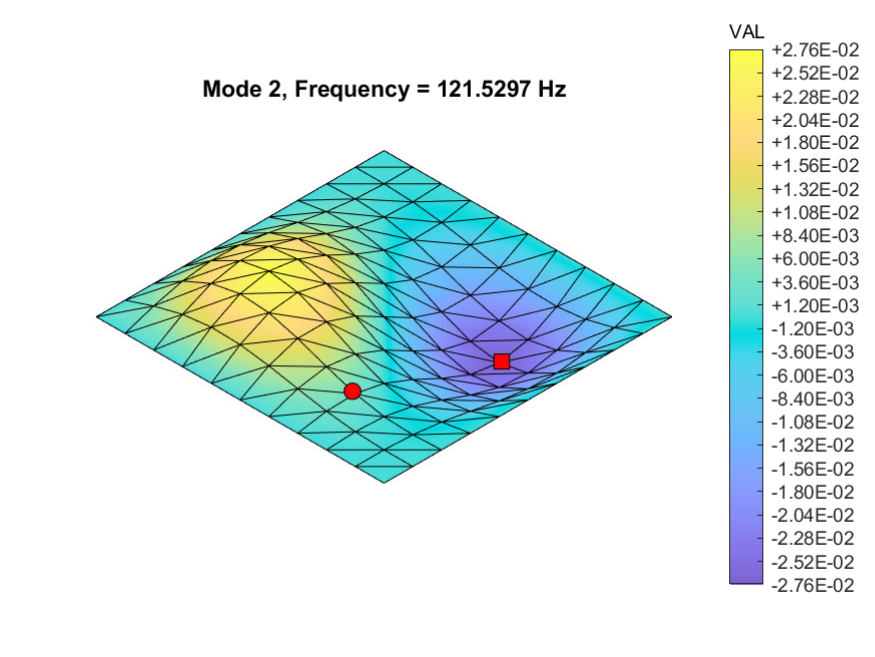}
\includegraphics[width=0.45\textwidth]{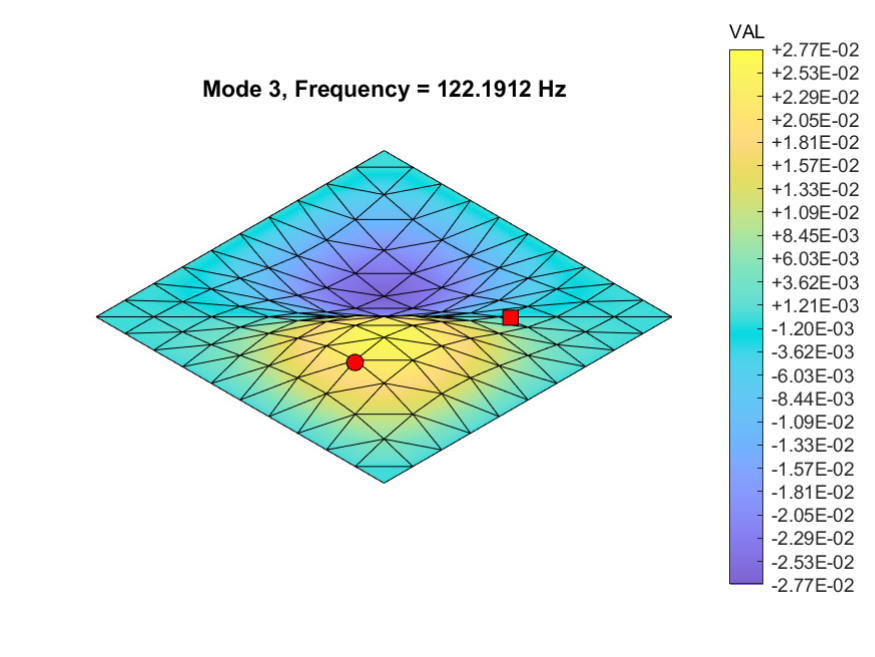}
\caption{Mode shapes of the second and third linear bending modes of the simply supported square plate~\cite{part-i}.}
\label{fig:plate_mode}
\end{figure}

With the Rayleigh damping assumption,  two pairs of complex conjugate eigenvalues associated with the second and the third bending mode are given as
\begin{gather}
\lambda_{3,4}\approx-1.7 \pm\mathrm{i}763.6\approx\pm\mathrm{i}\omega_2, \nonumber\\\lambda_{5,6}\approx-1.7 \pm\mathrm{i}767.7\approx\pm\mathrm{i}\omega_3.
\end{gather}
We take the four-dimensional spectral subspace corresponding to these eigenvalues as the master subspace $\mathcal{E}$ for SSM-based model reduction. In the following computations, we set $\Omega_\mathrm{lb}=0.95\mathrm{Im}(\lambda_3)=725.4$ and $\Omega_\mathrm{ub}=1.15\mathrm{Im}(\lambda_3)=878.1$. For the forcing amplitude, we take $\epsilon_\mathrm{lb}=0.01$ and $\epsilon_\mathrm{ub}=1$. 

Similarly to the previous example, we compare the FRCs obtained from the TV-SSM solution~\eqref{eq:ssm-time-varying} and the TI-SSM solution~\eqref{eq:ssm-time-independent} to conclude that the TI-SSM solution sufficiently approximates the converged FRC. We set $\epsilon=\epsilon_\mathrm{ub}$ and compute the corresponding FRC, where the predictions from TI-SSM agree with those from TV-SSM as shown in Fig.~\ref{fig:plate_plate_ti_vs_tv}. Here, we use an $\mathcal{O}(5)$ expansion for the SSM and its reduced dynamics to obtain predictions with sufficient accuracy ~\cite{part-i}.

\begin{figure}[!ht]
\centering
\includegraphics[width=.45\textwidth]{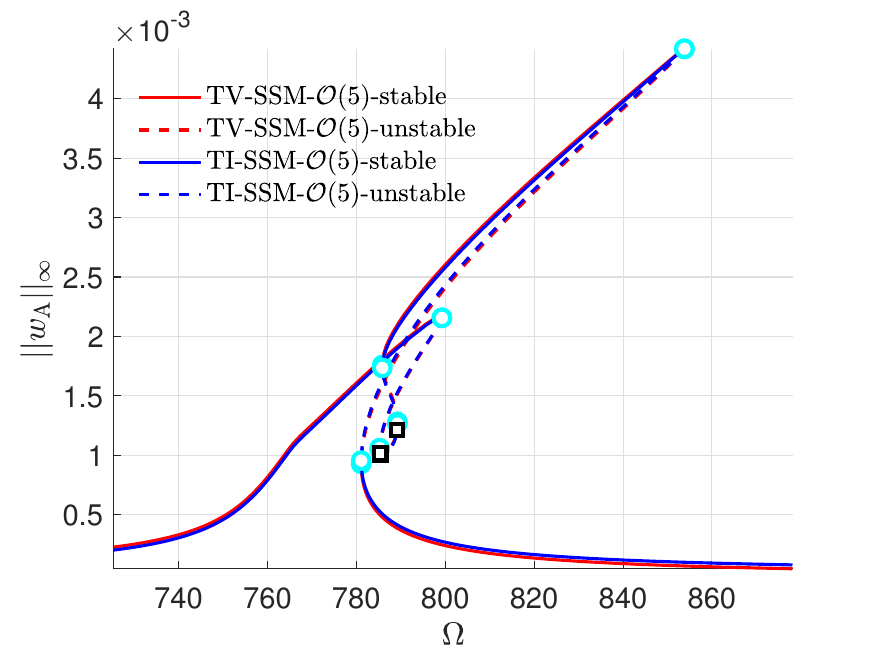}
\includegraphics[width=.45\textwidth]{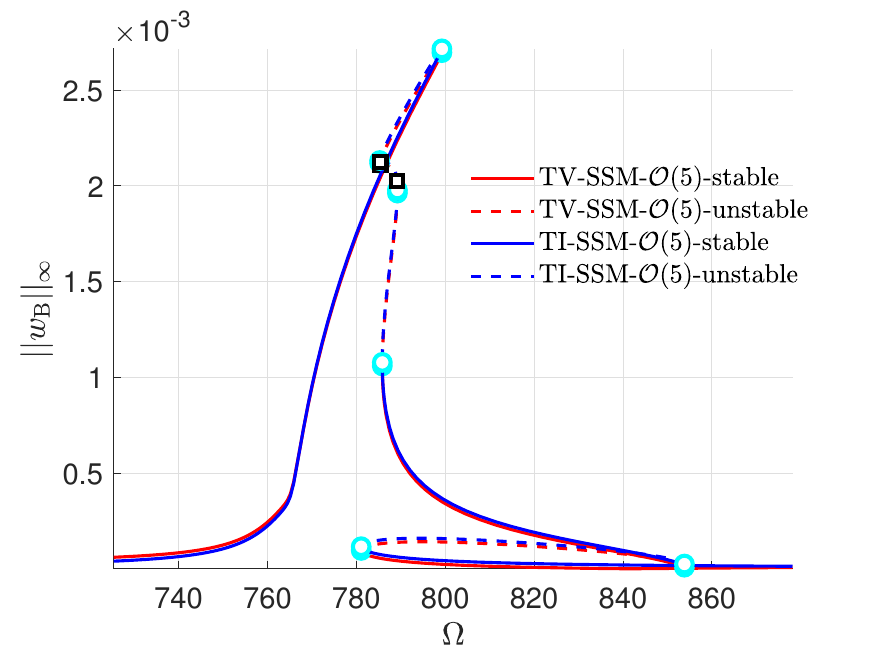}
\caption{Frequency response curves (amplitude of
periodic responses for the transverse displacement $w$ at points A and B) for the von K\'arm\'an plate discretized with 606 DOFs. Here $\epsilon=\epsilon_\mathrm{ub}=1$. TV-SSM and TI-SSM correspond to~\eqref{eq:ssm-time-varying} and~\eqref{eq:ssm-time-independent}, respectively.}
\label{fig:plate_plate_ti_vs_tv}
\end{figure}

Next, we compute the FRS for three amplitude objectives in~\eqref{eq:obj-plate}. As our analytic results for FRS predictions are only available for two-dimensional SSMs, we employ numerical continuation of fixed points of the reduced dynamics to compute the FRS associated with the four-dimensional SSM (see Sect.~\ref{sec:frs-ssm}). The FRS obtained is shown in Fig.~\ref{fig:plate_frs}. We observe that the FRS has a complicated geometry with self-intersections for sufficiently large values of $\epsilon$ (also see Fig.~\ref{fig:plate_plate_ti_vs_tv}). Here, the FRS is approximated via roughly 10,000 polygons, adaptively determined by the multidimensional continuation algorithm.

\begin{sloppypar}
We now validate the FRS obtained by our SSM-based ROM against the FRCs of the full system for the samples of the forcing amplitude $\epsilon\in\{0.25,0.5,0.75,1\}$. The collocation technique employed for full-system simulations in the previous example is not feasible here due to the large dimensionality of the full system~\cite{part-i}. As an alternative, we use a shooting method combined with parameter continuation to calculate the FRCs of the full system. Specifically, we use a~\textsc{coco}-based shooting toolbox~\cite{coco-shoot} for validation, where the Newmark scheme is used for the forward simulation. We use 1,000 time steps per excitation period for numerical integration~\cite{part-i}. At the same time, we increase the maximum continuation step size from the default value of 0.5 to 50 so that we can obtain the FRCs of the full system in a reasonable amount of time. As seen in Fig.~\ref{fig:plate_frs}, the FRCs obtained for the full system agree with the FRS predicted by our SSM-based ROM. The total computational time for the four FRCs in Fig.~\ref{fig:plate_frs}  is approximately 31 days. In contrast, we obtain the entire FRS in just 2 hours. %Thus, we obtain a significant speed-up from the SSM-based model reduction.
\end{sloppypar}

\begin{figure*}[!ht]
\centering
\includegraphics[width=.45\textwidth]{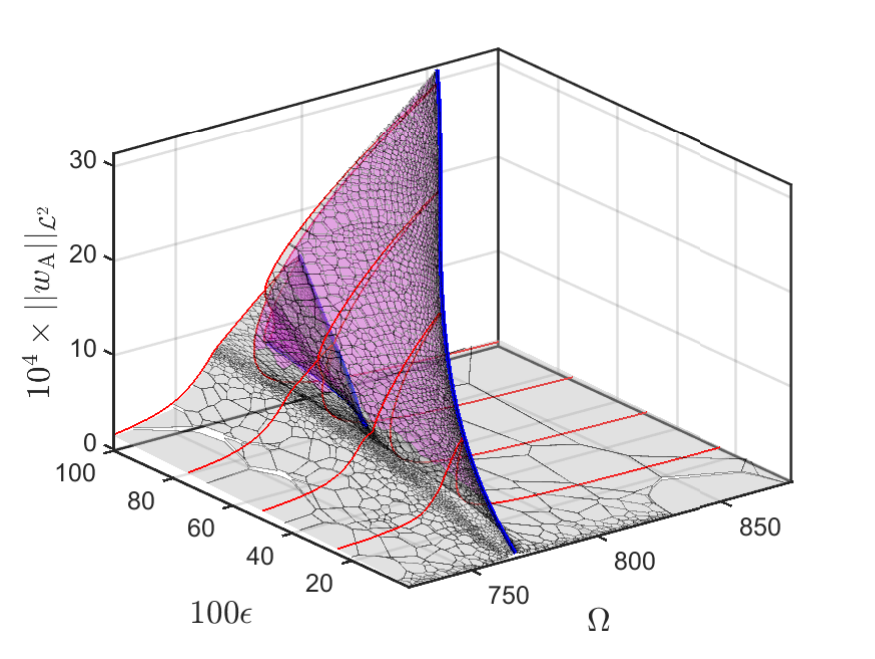}
\includegraphics[width=.45\textwidth]{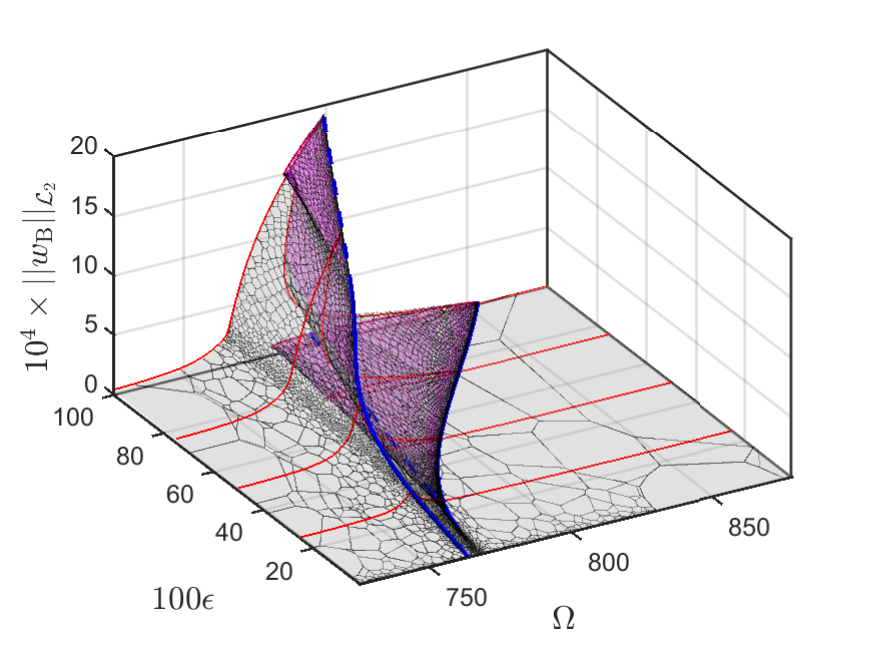}\\
\includegraphics[width=.45\textwidth]{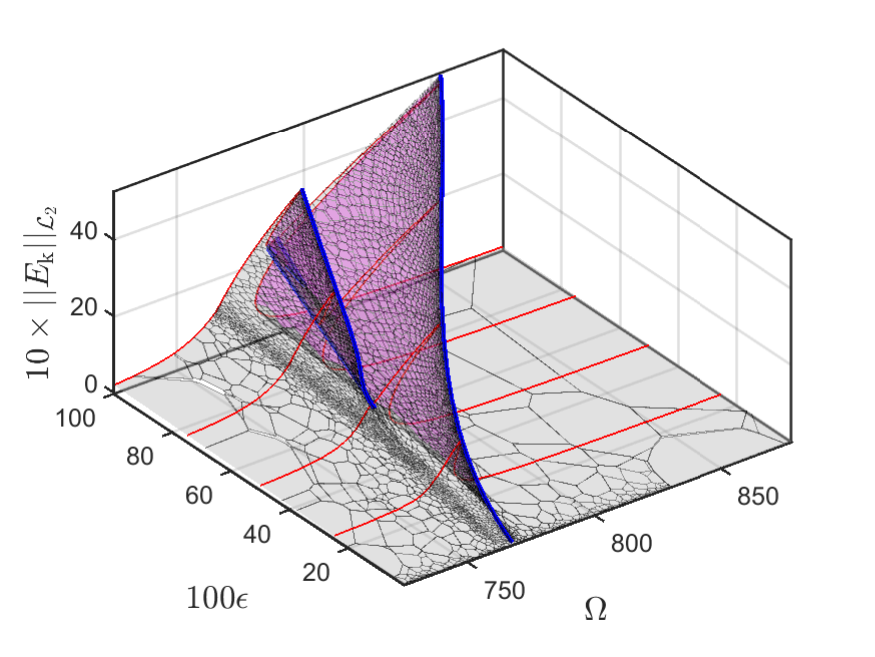}
\caption{FRSs and their ridges and trenches of the square plate. The top left and right panels give the FRS in terms of $||w_\mathrm{A}||_{\mathcal{L}^2}$ and $||w_\mathrm{B}||_{\mathcal{L}^2}$, while the lower panel gives the FRS for $||E_k||_{\mathcal{L}^2}$. Some sampled FRCs (red lines) of the full system are also plotted here for validation. We have rescaled $\epsilon$ and the amplitudes properly such that the ranges of all these variables have comparable magnitudes. This is important for the computation of these FRSs.}
\label{fig:plate_frs}
\end{figure*}

Now we compute the ridges and trenches on the FRS. As we will see, we obtain these curves in 4 minutes via the SSM-based ROM, which is significantly less than the two hours needed for the SSM-based FRS computation. To locate the ridges and trenches on the FRS, we first take $||w_\mathrm{A}||_{\mathcal{L}^2}$ as an optimization objective, initialize $\epsilon=\epsilon_\mathrm{ub}$ and apply the solution procedure established in~\ref{sec:sol-p3}. Along the FRC for $\epsilon=\epsilon_\mathrm{ub}$, three extrema are detected (two local maxima and one local minimum), as shown by the blue markers in the top-right panel of Fig.~\ref{fig:plate_ridge_3d}. Among the two maxima, we refer to the one with a higher value of $||w_\mathrm{A}||_{\mathcal{L}^2}$ as the global maximum and the other one as the local maximum. The trench emanating from the minimum and the ridge associated with the local maximum merge when $\epsilon$ decreases to $\epsilon \approx 0.525$. We further observe that the merger of a ridge and a trench results in their disappearance beyond the point of merger. On the other hand, the ridge emanating from the global maximum persists for $\epsilon\in(0,1]$ and the response amplitude along this ridge converges to zero when $\epsilon\to0$, as shown in the top-left panel of Fig.~\ref{fig:plate_ridge_3d}.

\begin{figure*}[!ht]
\centering
\includegraphics[width=.45\textwidth]{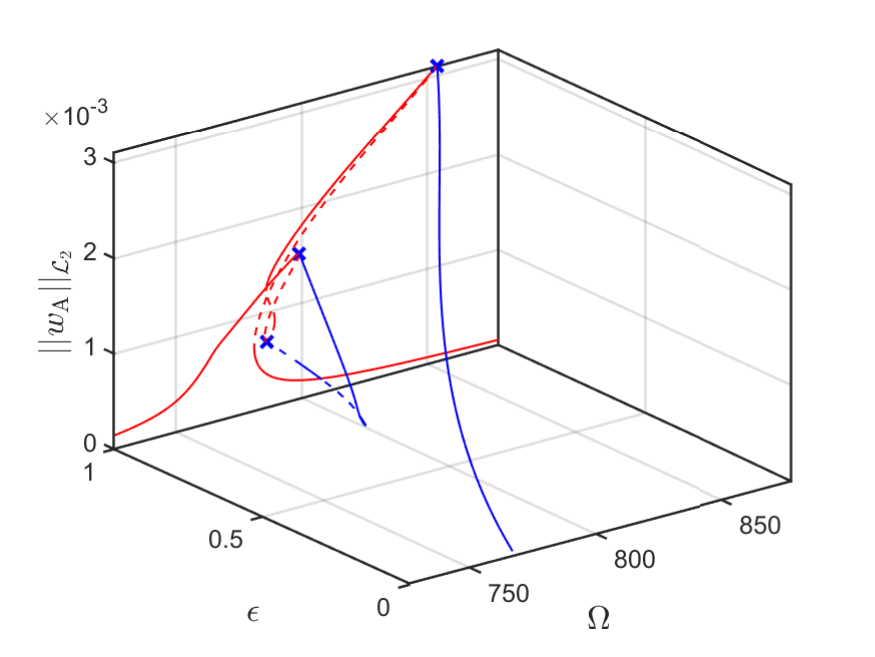}
\includegraphics[width=.45\textwidth]{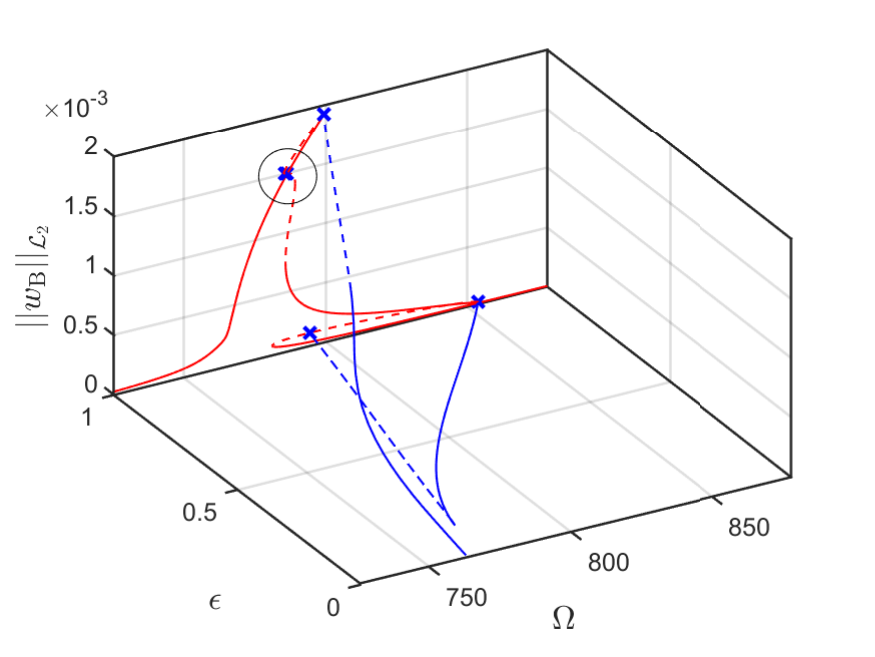}\\
\includegraphics[width=.45\textwidth]{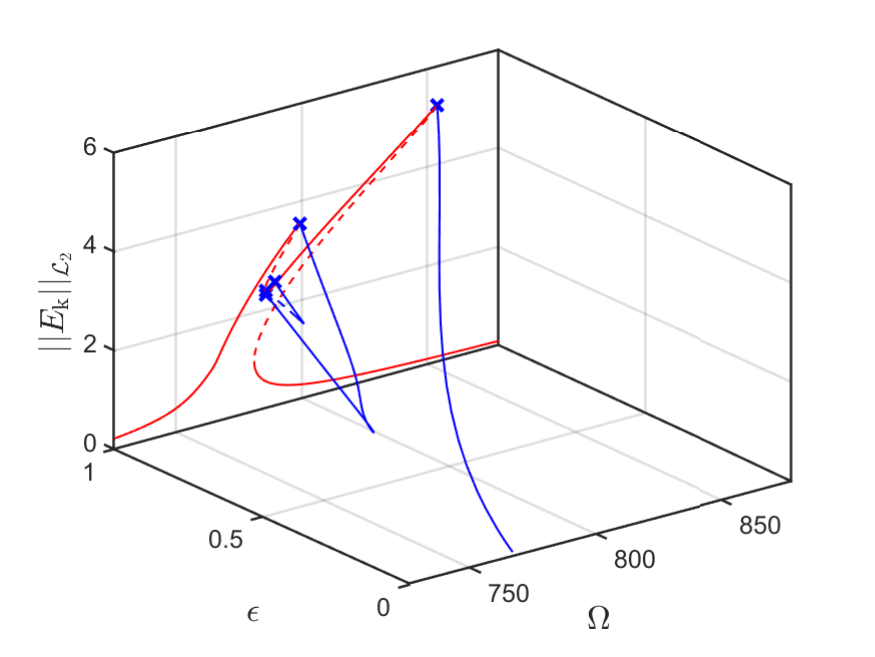}
\caption{FRCs (red lines) and resulting ridges and trenches (blue lines) in the FRS of the square plate in terms of different characterizations of response amplitude: $||w_\mathrm{A}||_{\mathcal{L}^2}$ (top-left panel), $||w_\mathrm{B}||_{\mathcal{L}^2}$ (top-right panel) and $||E||_{\mathcal{L}^2}$ (lower-panel). Here red lines are FRCs for $\epsilon=\epsilon_\mathrm{ub}$ and blue markers are local extrema on the FRCs.}
\label{fig:plate_ridge_3d}
\end{figure*}

Next, we take $||w_\mathrm{B}||_{\mathcal{L}^2}$ as the optimization objective and again apply the solution procedure established in~\ref{sec:sol-p3}. In this case, five extrema are detected along the FRC for $\epsilon=\epsilon_\mathrm{ub}$, as seen in the top-right panel of Fig.~\ref{fig:plate_ridge_3d}. Here, we observe two extrema that are close to each other within the circle. These two extrema quickly merge and disappear when $\epsilon$ decreases below $\epsilon=\epsilon_\mathrm{ub}=1$. We also observe a ridge merging with a trench as $\epsilon\to0.1332$. Similarly to the previous case, the ridge emanating from the global maximum persists for $\epsilon\in(0,1]$ and the response amplitude on this ridge converges to zero when $\epsilon\to0$.

Finally, we take $||E_\mathrm{k}||_{\mathcal{L}^2}$ as the optimization objective and repeat the analysis above. We again detect five extrema along the FRC for $\epsilon=\epsilon_\mathrm{ub}$, as seen in the lower panel of Fig.~\ref{fig:plate_ridge_3d}. We observe that a ridge and a trench merge near $\epsilon=0.8466$. Another ridge also merges with the second trench near $\epsilon=0.4905$. Once again, the ridge emanating from the global maximum persists for $\epsilon\in(0,1]$.

We conclude from the above discussion that the geometry of the FRS and its ridges and trenches depend on the choice of the optimization objective, especially, for mechanical systems with internal resonance. Indeed, for internally resonant systems, the FRSs constructed for modal response amplitudes may have significantly different features. The modal contributions to the FRSs constructed for some physical response amplitudes can change depending on the choice of the optimization objective. 

% The FRS in physical coordinates is resulted from the contributions of responses of all these modes. Depending on the choice of physical coordinates, each mode can have different levels of contributions. Consequently, the FRS.

%Therefore, the geometry of the FRSs in physical coordinates depends on the choice of the optimization objective that is defined in terms of the physical coordinates.

The computational times for generating the ridges and trenches in the three panels of Fig.~\ref{fig:plate_ridge_3d} are 60 seconds, 65 seconds, and 155 seconds, respectively. Indeed, these times are significantly less than the two hours required to obtain the entire FRS in Fig.~\ref{fig:plate_frs}.  We observe in Fig.~\ref{fig:plate_frs} that the ridges and trenches provide a skeleton of the FRS.

\subsection{A shallow shell with 1:2 internal resonance}

As our final example, we consider the nonlinear vibrations of a shallow-arc structure~\cite{part-i}, shown in Fig.~\ref{fig:shell_mesh}. Here, the shell is simply supported at the two opposite edges aligned along the $y-$axis in Fig.~\ref{fig:shell_mesh}.

\begin{figure}[!ht]
\centering
\includegraphics[width=0.4\textwidth]{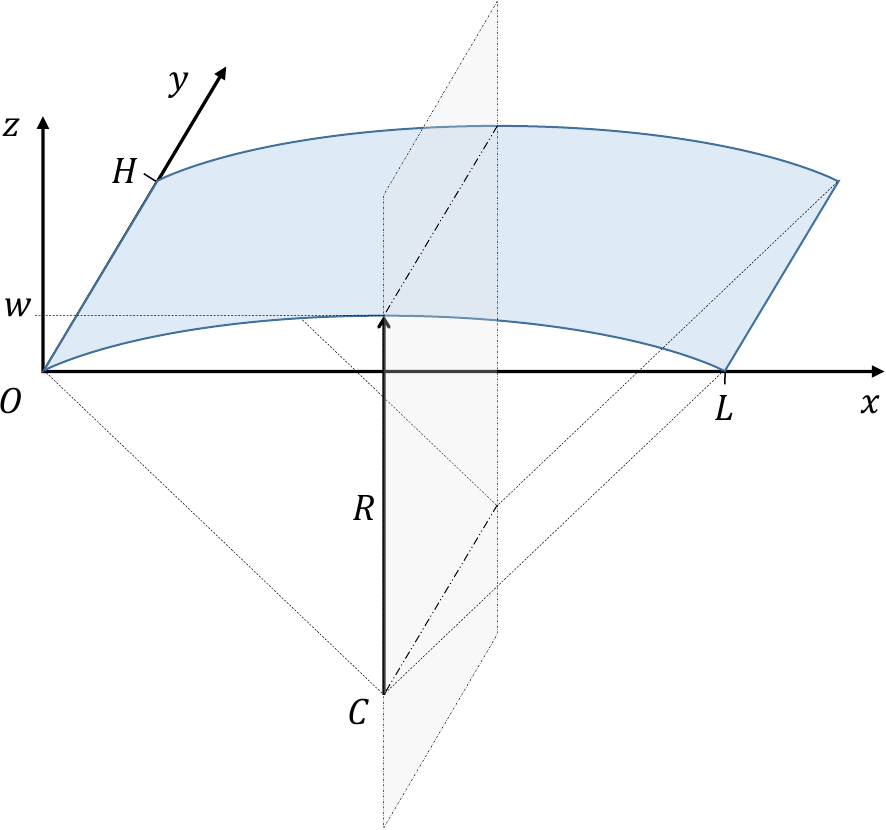}
\caption{The schematic of a shallow shell structure~\cite{jain2022compute,part-i}.}
\label{fig:shell_mesh}
\end{figure}

\begin{sloppypar}
The geometric and material properties of this shell can be found in~\cite{part-i}. We use the same finite-element model as in~\cite{part-i}. Specifically, the discrete model has 400 elements and 1,320 DOFs, resulting in a 2,640-dimensional phase space. With the chosen curvature, the first two bending modes of this structure admit a 1:2 internal resonance~\cite{part-i}. In particular, the eigenvalues of the first two pairs of modes of the discrete model are given by~\cite{part-i}
\begin{equation}
    \lambda_{1,2}=-0.30\pm\mathrm{i}149.22,\quad\lambda_{3,4}=-0.60\pm\mathrm{i}298.78.
\end{equation}

We apply a concentrated load $\epsilon100\cos\Omega t$ in the $z-$ direction at mesh node A with $(x,y)=(0.25L,0.5H)$. We are concerned with the forced response in terms of the $z$-displacements of node A and node B, where node B is located at $(x,y)=(0.5L,0.5H)$. We set $\Omega_\mathrm{lb}=0.92\mathrm{Im}(\lambda_1)=137.2857$, $\Omega_\mathrm{ub}=1.07\mathrm{Im}(\lambda_1)=159.6693$, $\epsilon_\mathrm{lb}=0.001$, and $\epsilon_\mathrm{ub}=0.1$ to extract the FRS around the first mode.
\end{sloppypar}

Similarly to the previous two examples, we compare the FRCs obtained from the TV-SSM solution~\eqref{eq:ssm-time-varying} and the TI-SSM solution~\eqref{eq:ssm-time-independent} to conclude that the TI-SSM solution sufficiently approximates the converged FRC. Specifically, we set $\epsilon=\epsilon_\mathrm{ub}$ and calculate the FRC using the TI-SSM and TV-SSM solutions, which match closely, as shown in Fig.~\ref{fig:shell_ti_vs_tv}. Therefore, we use the TI-SSM solutions to make faster FRS predictions in this example.

\begin{figure}[!ht]
\centering
\includegraphics[width=.45\textwidth]{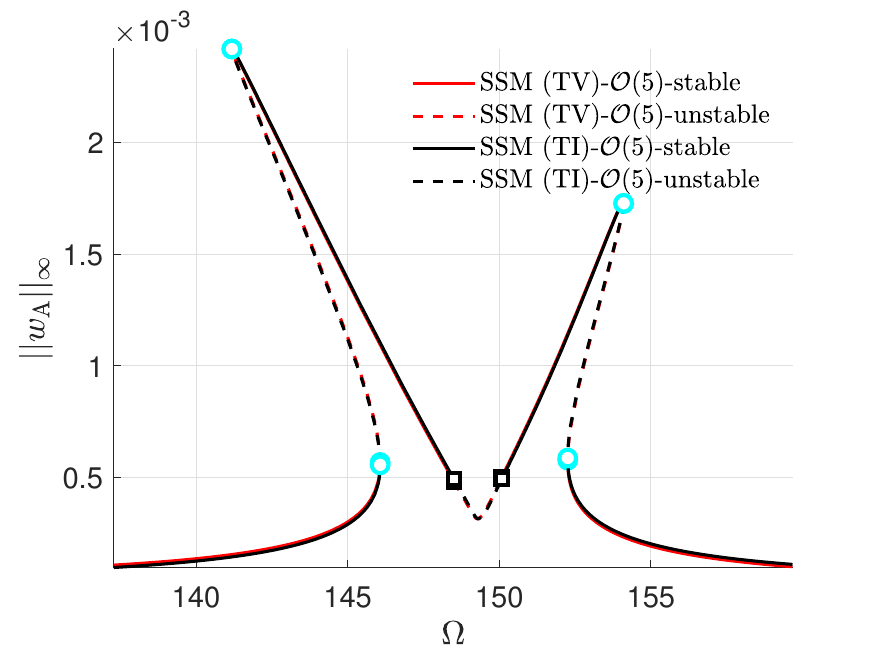}
\includegraphics[width=.45\textwidth]{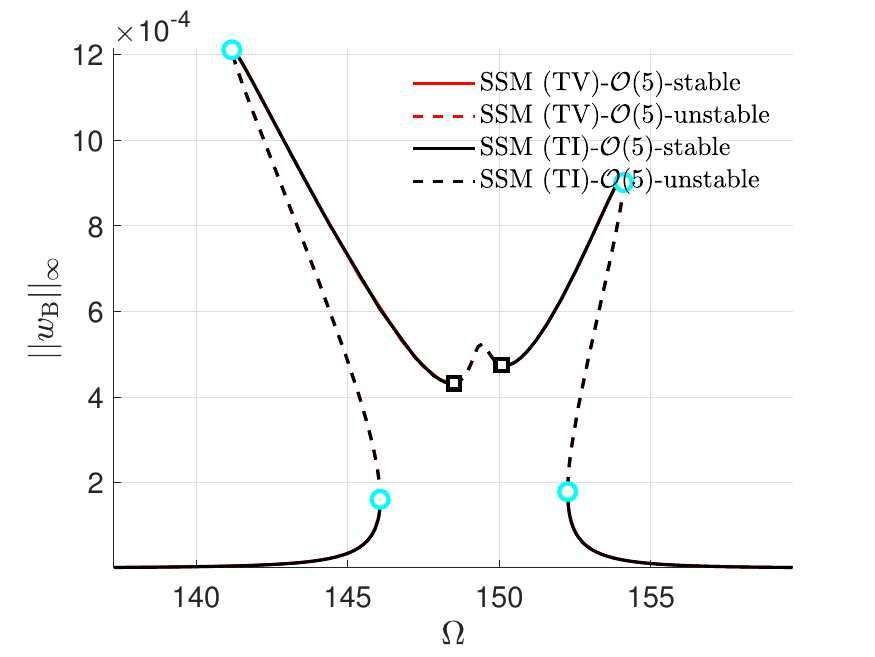}
\caption{FRCs (amplitude of periodic responses for the transverse displacement $w$ at points A and B) for the shallow shell discretized with 1320 DOFs. Here $\epsilon=\epsilon_\mathrm{ub}=0.1$. TV-SSM and TI-SSM correspond to~\eqref{eq:ssm-time-varying} and~\eqref{eq:ssm-time-independent}, respectively.}
\label{fig:shell_ti_vs_tv}
\end{figure}

For FRS computation, we use the response amplitude objectives $||w_\mathrm{A}||_{\mathcal{L}^2}$, $||w_\mathrm{B}||_{\mathcal{L}^2}$ (see~\eqref{eq:obj-plate} for detailed definitions), as well as $||w_\mathrm{A}||_{\mathcal{L}^\infty}$ and $||w_\mathrm{B}||_{\mathcal{L}^\infty}$ . Similarly to the previous example, we obtain the FRS shown in Fig.~\ref{fig:shell_frs} via the two-dimensional continuation of fixed points of the SSM-based ROM. We again observe that the geometry of the FRS depends on the choice of our amplitude objective. Indeed, in Fig.~\ref{fig:shell_frs}, we observe a local ridge near the primary trench in the FRS of point B for sufficiently large values of $\epsilon$, which is different to the FRS of point A.

%Comparing the upper and lower panels in Fig.~\ref{fig:shell_frs}, we observe that the geometry of the FRS is independent of the norms used to characterize the response amplitude. However, upon comparing the left and corresponding right panels, we observe that the FRS of point B is different from that of point A. In particular, a local ridge near the primary trench is observed in the FRS of point B when $\epsilon$ is large enough.

\begin{figure*}[!ht]
\centering
\includegraphics[width=.45\textwidth]{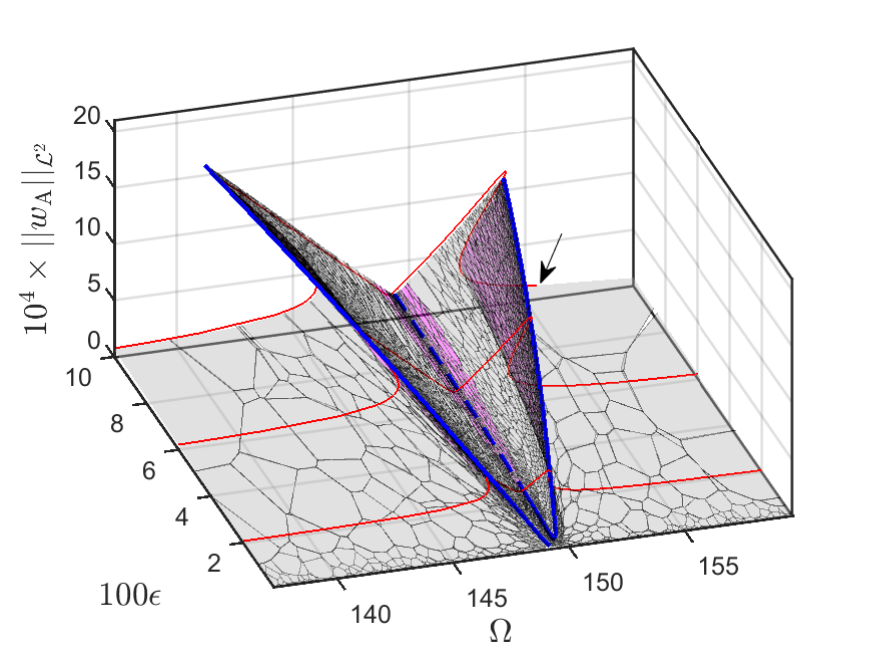}
\includegraphics[width=.45\textwidth]{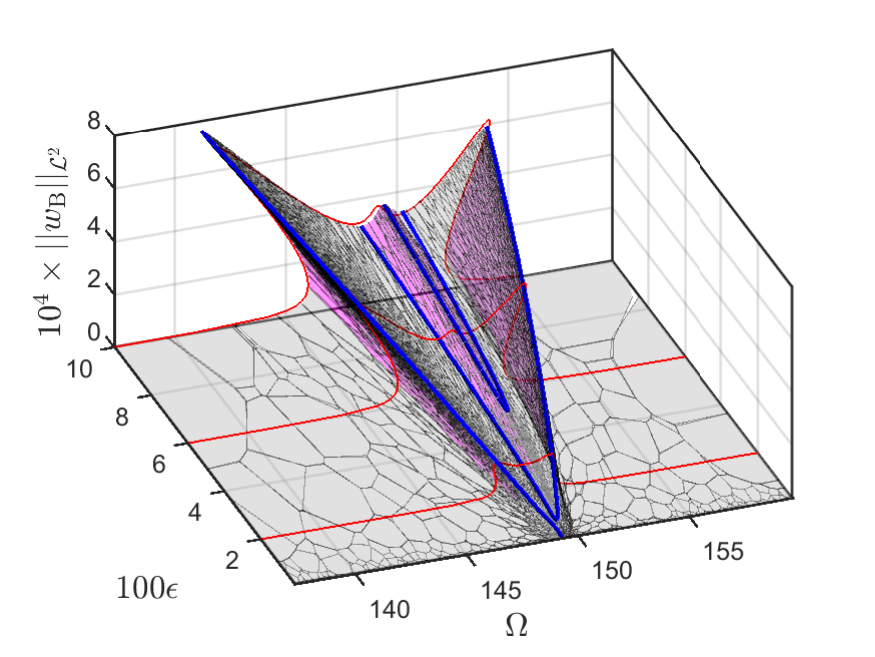}\\
\includegraphics[width=.45\textwidth]{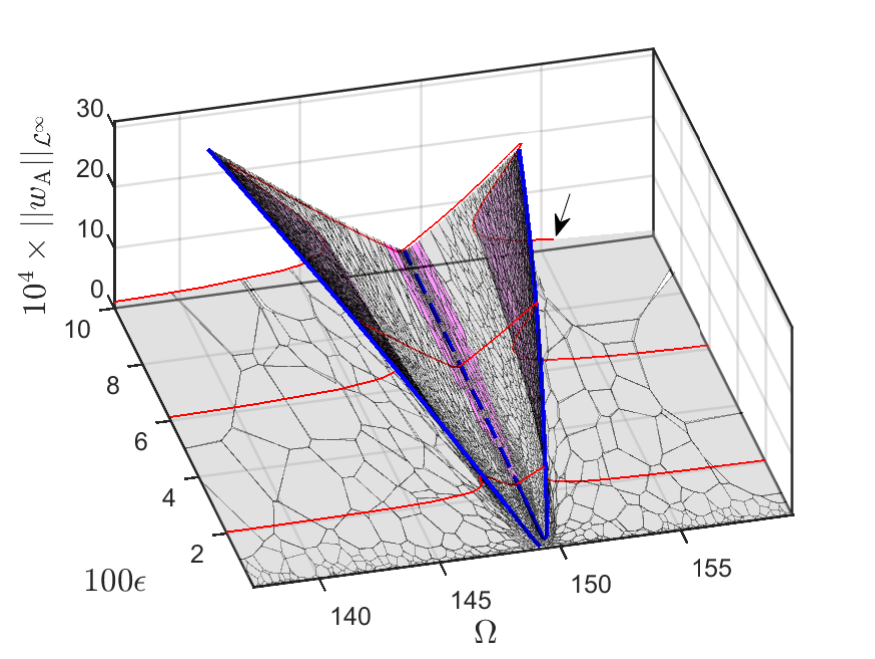}
\includegraphics[width=.45\textwidth]{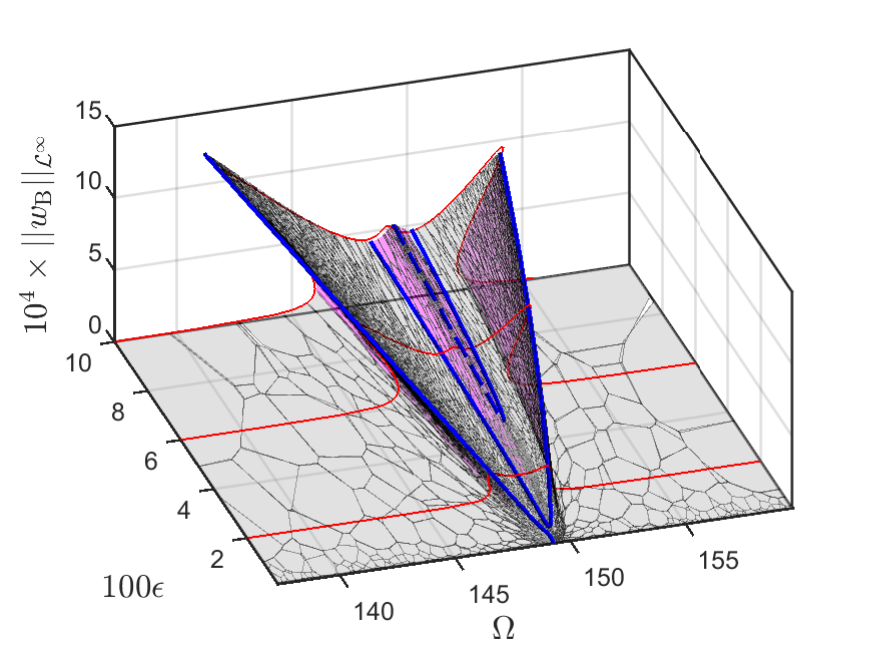}
\caption{FRS and its ridges and trenches (blue lines) of the shallow shell. The top left and right panels give the FRS in terms of $||w_\mathrm{A}||_{\mathcal{L}^2}$ and $||w_\mathrm{B}||_{\mathcal{L}^2}$, while the lower left and right panels give the FRS in terms of $||w_\mathrm{A}||_{\mathcal{L}^\infty}$ and $||w_\mathrm{B}||_{\mathcal{L}^\infty}$. Some sampled FRCs (red lines) of the full system are provided for the purpose of validation. We have rescaled $\epsilon$ and the amplitudes properly such that the ranges of all these variables have comparable magnitudes.}
\label{fig:shell_frs}
\end{figure*}

We now validate the FRS obtain via our SSM-based ROM against the FRCs of the full system for the forcing amplitude samples $\epsilon\in\{0.02, 0.06, 0.1\}$. Similarly to the previous example, we use the shooting method combined with parameter continuation~\cite{coco-shoot} to calculate the FRCs of the full system. As seen in Fig.~\ref{fig:shell_frs}, the three sampled FRCs lie close to the FRSs, which validates the accuracy of the SSM-based predictions. The computational times to obtain the three FRCs with $\epsilon=0.02, 0.06$ and 0.1 are approximately 90, 146, and 180 hours, respectively. We have set the computational time limit for each continuation run to be 180 hours. It turns out that the continuation run for the FRC with $\epsilon=0.1$  was terminated near  $\Omega=156$ rad/s (see the arrows in the left panels of Fig.~\ref{fig:shell_frs}) as it reached the set time limit. On the other hand, we recall that the computational time to obtain the entire FRS via the SSM-based ROM is only about 1.5 hours.% So we again obtain a significant speed-up gain from the SSM-based model reduction.

Next, we aim to obtain a skeleton of the FRS by computing the ridges and trenches of the FRS at a fraction of the cost associated with FRS computation. Taking $||w_\mathrm{A}||_{\mathcal{L}^2}$ as the optimization objective, initialize $\epsilon=\epsilon_\mathrm{ub}$, we apply the solution procedure in Sect.~\ref{sec:sol-p3} to obtain the ridges and trenches. We first compute the FRC for $\epsilon=\epsilon_\mathrm{ub}$, where two local maxima and one local minimum are detected, as shown by the blue markers in the upper-left panel of Fig.~\ref{fig:shell_ridge_3d}. Among the two maxima, the one with $\Omega\approx141$ is the global maximum, while the one with $\Omega\approx154$ is a local maximum. The trench associated with the local minimum merges with the ridge emanating from the local maximum as $\epsilon\to0$. The remaining ridge persists as the global maximum of FRCs in the computational domain and the response amplitude $||w_\mathrm{A}||_{\mathcal{L}^2}$ converges to zero along this ridge as $\epsilon\to0$.

\begin{figure*}[!ht]
\centering
\includegraphics[width=.45\textwidth]{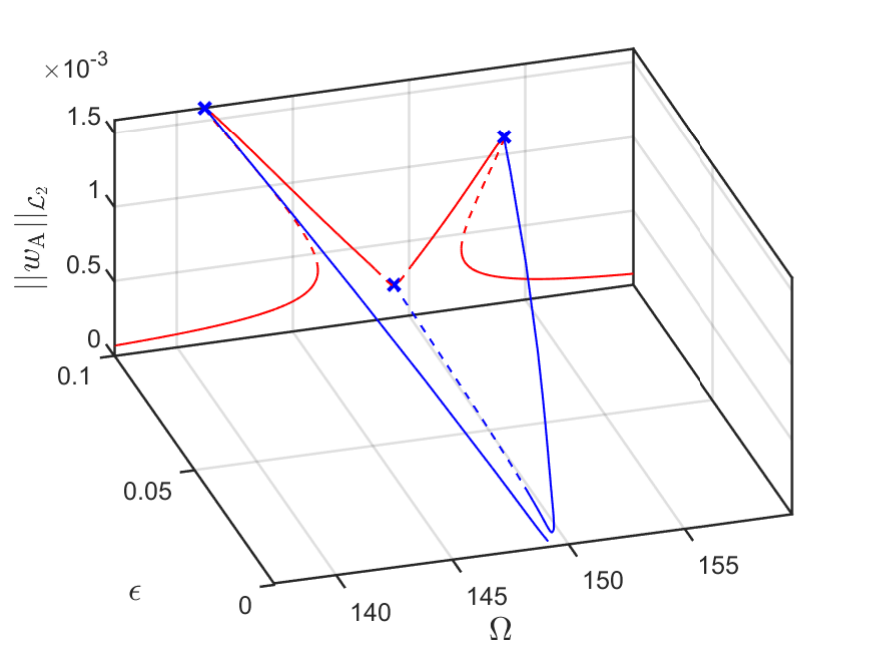}
\includegraphics[width=.45\textwidth]{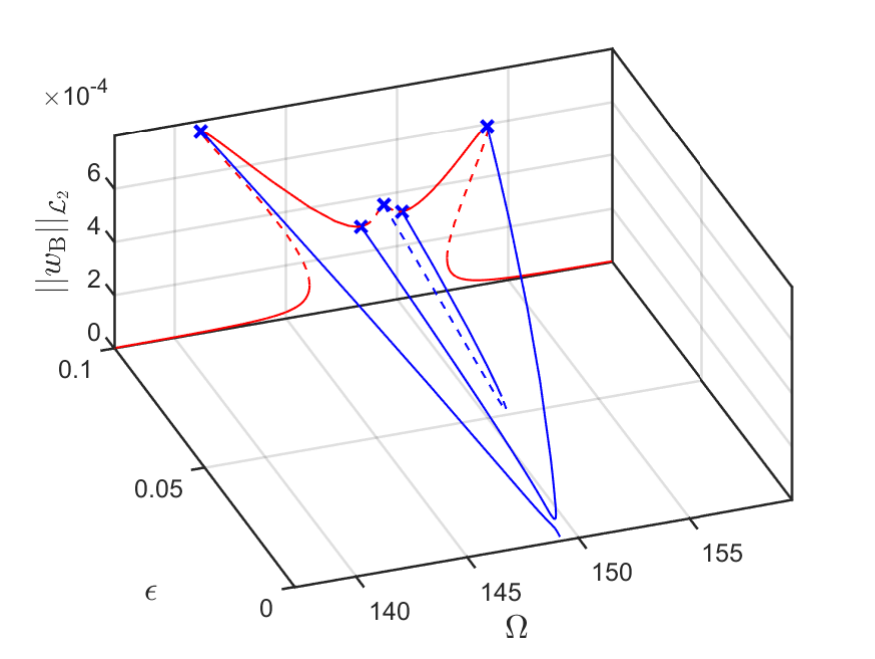}\\
\includegraphics[width=.45\textwidth]{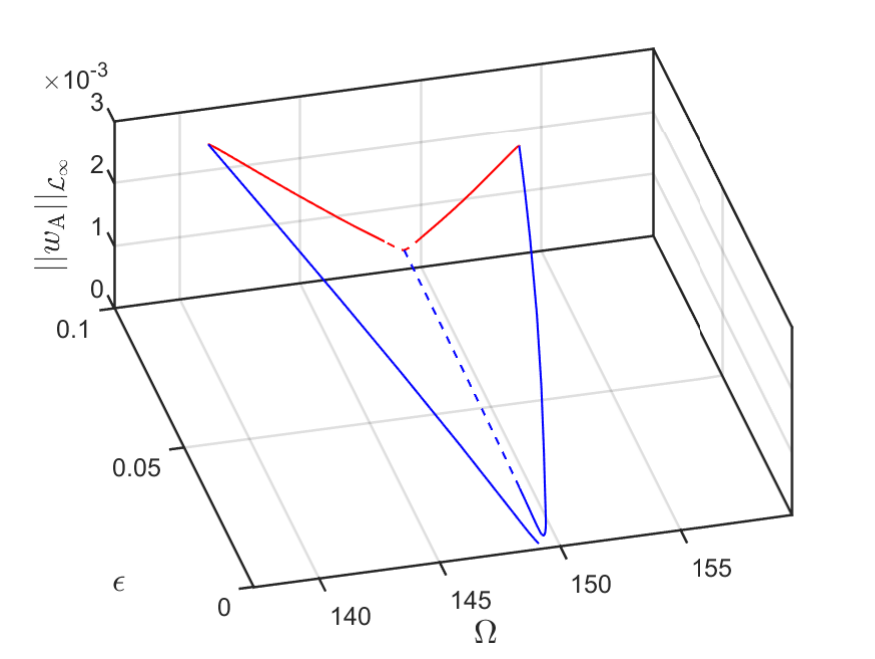}
\includegraphics[width=.45\textwidth]{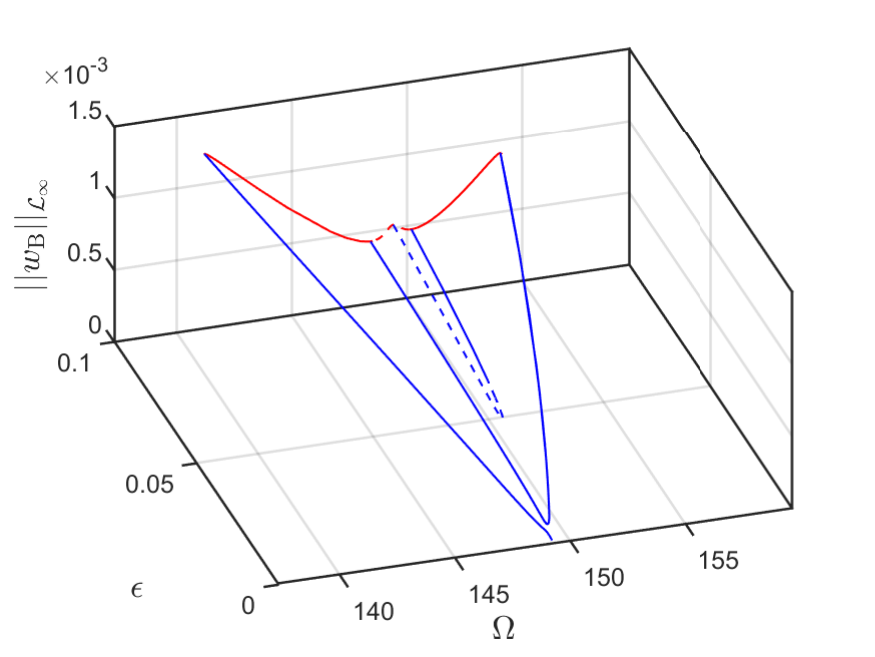}
\caption{FRCs (red lines) and resulting ridges and trenches (blue lines) in the FRS of the shallow shell. The upper left and right panels show the results for $||w_\mathrm{A}||_{\mathcal{L}^2}$ and $||w_\mathrm{B}||_{\mathcal{L}^2}$, while the lower left and right panels present the results for $||w_\mathrm{A}||_{\mathcal{L}^\infty}$ and $||w_\mathrm{B}||_{\mathcal{L}^\infty}$.}
\label{fig:shell_ridge_3d}
\end{figure*}

We now consider the optimization objective to $||w_\mathrm{B}||_{\mathcal{L}^2}$ and repeat the above computation of ridges and trenches. For this optimization objective, five extrema (three local maxima and two local minima) are detected along the FRC at $\epsilon=\epsilon_\mathrm{ub}$. As seen in the upper-right panel of Fig.~\ref{fig:shell_ridge_3d}, a ridge-trench pair around $\Omega=148$ merges to disappear near $\epsilon=0.035$. Another ridge-trench pair merges around $\epsilon=0.006$. Similarly to the previous optimization objective, the ridge corresponding to the global maximum persists in the computational domain, and the response amplitude along this ridge converges to zero as $\epsilon\to0$.

Finally, we consider the amplitude objectives in the $\mathcal{L}^\infty$ norm instead of those in the $\mathcal{L}^2$ norm earlier. We demonstrate the effectiveness of the solution procedure proposed in Sect.~\ref{sec:sol-p4} for the optimization objective based on the $\mathcal{L}^\infty$ norm. For the objective $||w_\mathrm{A}||_{\mathcal{L}^\infty}$, the ridges and trenches are shown in the lower-left panel of Fig.~\ref{fig:shell_ridge_3d}. We observe that the FRC in this panel is terminated at two points that are very close to the intersection points of the two ridges and the FRC. These two terminations occur due to a saddle-node bifurcation on the FRC (see Remark~\ref{rmk-sn}). For the optimization objective $||w_\mathrm{B}||_{\mathcal{L}^\infty}$, we obtain the ridges and trenches shown in the lower-right panel of Fig.~\ref{fig:shell_ridge_3d}. By comparing the upper and lower panels in Fig.~\ref{fig:shell_ridge_3d}, we observe that the ridges and trenches are qualitatively similar for the optimization objectives in the $\mathcal{L}^2$ and  $\mathcal{L}^\infty$ norms.

The computational times for generating the ridges and trenches in the four panels of Fig.~\ref{fig:shell_ridge_3d} are 71, 75, 77, and 81 seconds, which is significantly less than the 1.5 hours used to generate the entire FRS in Fig.~\ref{fig:shell_frs}. We observe in Fig.~\ref{fig:shell_frs} that the ridges and trenches provide a skeleton of the FRS. Therefore, to characterize the forced response, it is worthwhile to compute only the ridges and trenches of the FRS instead of the entire FRS.  

% COMPUTATIONAL TIME L2 NORM 71 and 75, 77 and 81

\section{Conclusion}

We have developed a new, SSM-based approach for computing the forced response surface (FRS) of harmonically excited high-dimensional mechanical systems including those with internal resonances. To this end, we constructed low-dimensional reduced-order models (ROMs) for such systems using spectral manifolds (SSMs). We computed the FRS as a two-dimensional manifold of fixed points of the SSM-based ROMs. For systems without internal resonance, we obtained analytic expressions for the FRS via the SSM-based ROMs. For general systems with possible internal resonances, we have used multidimensional manifold continuation to cover the FRS as a manifold of fixed points of the SSM-based ROMs. Since manifold covering may still be a demanding task, we further illustrated how to extract the ridges and trenches on the FRS directly without computing the entire FRS. We achieved this extraction via a successive continuation technique applied to an augmented continuation problem that is derived from the first-order necessary conditions of appropriately defined optimization problems on the SSM-based ROMs. These ridges and trenches serve as a skeleton of the FRS and  their computation provides a fast characterization of the FRS.

We have also demonstrated the accuracy and efficiency of the SSM-based reduction method using three examples. In the first example, a 50 DOF cantilever beam with a nonlinear support spring and damper was studied. This example demonstrates an analytic prediction of the FRS via a ROM based on a two-dimensional SSM. In addition, we showed that this FRS automatically detects the existence and bifurcation of isolas in the given range of forcing amplitude. We validated the accuracy of our SSM-based FRS prediction against 6 sampled FRCs of the full system via the collocation method. We also calculated the ridges and trenches of the SSM-based FRS in 31 seconds via the successive continuation method. We validated these ridges and trenches using the collocation method applied to an optimization problem for periodic orbits of the full system, which took approximately 1.5 days.

Next, we studied the forced response of a 606 DOF von K\'am\'an plate with 1:1 internal resonance. Here, we constructed a 4-dimensional SSM-based ROM for this internally resonant system. Based on the SSM-based ROM, we computed the ridges and trenches on the FRS in three minutes and the entire FRS in two hours. We validated our SSM-based predictions using a shooting method combined with parameter continuation to compute four sampled FRCs of the full system, which took more than 24 days. %So we again have a significant speed-up gain from the SSM-based model reduction. 

In the last example, we investigated the forced response of a 1,320-DOF shell structure with 1:2 internal resonance. We again constructed a 4-dimensional SSM-based ROM for the system taking into account the internal resonance. We computed the ridges and trenches on the FRS via our SSM-based ROM in less than 2 minutes, and the entire FRS in 1.5 hours. We validated the accuracy our SSM-based predictions  against three sampled FRCs of the full system calculated using the shooting method. The computational time for these three FRCs was more than 20 days, which again shows the significant speed-up gain from the SSM-based model reduction.

The computations performed in this study can be applied to systems with configuration constraints~\cite{li2022model} as well. While we have computed the FRS of periodic orbits in this work, it is instructive to extend this procedure to the computation of FRS of quasi-periodic orbits~\cite{part-ii}. %Furthermore, one can also extract the FRS via data-driven SSM-based ROMs~\cite{cenedese2022data}, which needs to be further studied and demonstrated.

\section*{Acknowledgement}
We are grateful to Mattia Cenedese for useful discussions and helpful comments on analytic solutions for forced response surfaces. ML gratefully acknowledges support by the National Natural Science Foundation of China (No. 12302014).

\appendix

\section{Theorem on periodic SSM}
\label{sec:app-ssm-existence}
\begin{theorem}
\label{th:SSM-existence-uniqueness}
Let $\spect(\mathcal{E}) = \{\lambda^\mathcal{E}_1,\bar{\lambda}^\mathcal{E}_1,\cdots,\lambda^\mathcal{E}_m,\bar{\lambda}^\mathcal{E}_m\}$ and define  $\spect(\boldsymbol{\Lambda})=\{\lambda_1,\cdots,\lambda_{2n}\}$.
Under the non-resonance condition
\begin{gather}
\boldsymbol{a}\cdot\mathrm{Re}(\boldsymbol{\lambda}^\mathcal{E})+\boldsymbol{b}\cdot\mathrm{Re}(\bar{\boldsymbol{\lambda}}^\mathcal{E})\neq \mathrm{Re}(\lambda_k),\nonumber\\
\forall\,\,\lambda_k\in\spect(\boldsymbol{\Lambda})\setminus\spect(\mathcal{E}),\nonumber\\
\forall\,\,\boldsymbol{a},\boldsymbol{b}\in\mathbb{N}_0^m,\,\,2\leq |\boldsymbol{a}+\boldsymbol{b}|\leq\Sigma(\mathcal{E}),
\end{gather}
where the \emph{absolute spectral quotient} $\Sigma(\mathcal{E})$ of $\mathcal{E}$ is defined as
\begin{equation}
\Sigma(\mathcal{E}) = \mathrm{Int}\left(\frac{\min_{\lambda\in\spect(\boldsymbol{\Lambda})}\mathrm{Re}\lambda}{\max_{\lambda\in\spect(\mathcal{E})}\mathrm{Re}\lambda}\right).
\end{equation}
Assume further that $r>\Sigma(\mathcal{E})$. Then for any $\epsilon>0$ small enough, the following hold for system~\eqref{eq:full-first}:
\begin{enumerate}[label=(\roman*)]
\item There exists a $2m$-dimensional, time-periodic, class $C^r$ SSM $\mathcal{W}(\mathcal{E},\Omega t)$ that depends smoothly on the parameter $\epsilon$.
\item The SSM $\mathcal{W}(\mathcal{E},\Omega t)$ is unique among all $C^{\Sigma(\mathcal{E})+1}$ invariant manifolds satisfy (i).
\item $\mathcal{W}(\mathcal{E},\Omega t)$ can be viewed as an embedding {of an open set in the reduced coordinates $(\boldsymbol{p},\phi)$} into the phase space of system~\eqref{eq:full-first} via the map
\begin{equation}
\boldsymbol{W}_{\epsilon}(\boldsymbol{p},\phi):{\mathbb{C}^{2m}}\times{S}^1\to\mathbb{R}^{2n} .
\end{equation}
\item There exists a polynomial {series} {$\boldsymbol{R}_{\epsilon}(\boldsymbol{p},\phi):\mathbb{C}^{2m}\times{S}^1\to \mathbb{C}^{2m}$} satisfying the invariance equation
\begin{align}
\label{eq:invariance}
& \boldsymbol{B}\left({D}_{\boldsymbol{p}}\boldsymbol{W}_{\epsilon}(\boldsymbol{p},\phi) \boldsymbol{R}_{\epsilon}(\boldsymbol{p},\phi)+{D}_{\phi}\boldsymbol{W}_{\epsilon}(\boldsymbol{p},\phi) \Omega\right)\nonumber\\
&=\boldsymbol{A}\boldsymbol{W}_{\epsilon}(\boldsymbol{p},\phi)+\boldsymbol{F}( \boldsymbol{W}_{\epsilon}(\boldsymbol{p},\phi))+\epsilon\boldsymbol{F}^{\mathrm{ext}}({\phi}),
\end{align}
such that the reduced dynamics on the SSM $\mathcal{W}(\mathcal{E},\Omega t)$ can be expressed as
\begin{equation}
\label{eq:red-dyn}
\dot{\boldsymbol{p}} = \boldsymbol{R}_\epsilon(\boldsymbol{p},\phi),\quad \dot{\phi}=\Omega.
\end{equation}
\end{enumerate}
\end{theorem}

\begin{proof}
This theorem is simply a restatement of Theorem 4 by Haller and Ponsioen~\cite{haller2016nonlinear}, which is based on more abstract results by Cabr\'e et al.~\cite{cabre2003parameterization-i,cabre2003parameterization-ii,cabre2005parameterization-iii} and Haro and de la Llave~\cite{haro2006parameterization,haro2006parameterization-num}.
\end{proof}

\section{Derivation of explicit responses amplitude}
\label{sec:appendix-th3}
In the case of polar coordinates, we substitute~\eqref{eq:polar-form} and obtain
\begin{align}
    \boldsymbol{q}^{\boldsymbol{c}}= & (\rho_1e^{\mathrm{i}(\theta_1+r_1\Omega t)})^{c_1}\cdots(\rho_me^{\mathrm{i}(\theta_m+r_m\Omega t)})^{c_m}\nonumber\\= &\boldsymbol{\rho}^{\boldsymbol{c}}e^{\mathrm{i}\boldsymbol{c}\cdot\boldsymbol{\theta}}e^{\mathrm{i}\boldsymbol{c}\cdot\boldsymbol{r}\Omega t},\\
    \bar{\boldsymbol{q}}^{\boldsymbol{d}}=& (\rho_1e^{-\mathrm{i}(\theta_1+r_1\Omega t)})^{d_1}\cdots(\rho_me^{-\mathrm{i}(\theta_m+r_m\Omega t)})^{d_m}\nonumber\\=&\boldsymbol{\rho}^{\boldsymbol{d}}e^{-\mathrm{i}\boldsymbol{d}\cdot\boldsymbol{\theta}}e^{-\mathrm{i}\boldsymbol{d}\cdot\boldsymbol{r}\Omega t},
\end{align}
and then $\boldsymbol{q}^{\boldsymbol{c}}\bar{\boldsymbol{q}}^{\boldsymbol{d}}=\boldsymbol{\rho}^{\boldsymbol{c}+\boldsymbol{d}}e^{\mathrm{i}(\boldsymbol{c}-\boldsymbol{d})\cdot\boldsymbol{\theta}}e^{\mathrm{i}(\boldsymbol{c}-\boldsymbol{d})\cdot\boldsymbol{r}\Omega t}$. Thus we have
\begin{align}
    \boldsymbol{W}(\boldsymbol{p})& =\sum_{(\boldsymbol{c},\boldsymbol{d})}\boldsymbol{w}_{(\boldsymbol{c},\boldsymbol{d})}\boldsymbol{q}^\mathrm{c}\bar{\boldsymbol{q}}^{\boldsymbol{d}}\nonumber\\
    & =\sum_{(\boldsymbol{c},\boldsymbol{d})}\boldsymbol{w}_{(\boldsymbol{c},\boldsymbol{d})}\boldsymbol{\rho}^{\boldsymbol{c}+\boldsymbol{d}}e^{\mathrm{i}(\boldsymbol{c}-\boldsymbol{d})\cdot\boldsymbol{\theta}}e^{\mathrm{i}(\boldsymbol{c}-\boldsymbol{d})\cdot\boldsymbol{r}\Omega t}\nonumber\\
    & =\sum_{\hat{r}_i}\boldsymbol{w}_{\hat{r}_i}^\mathrm{p}e^{\mathrm{i}{\hat{r}_i}\Omega t}
\end{align}
where
\begin{equation}
\label{eq:what-ri-p}
\boldsymbol{w}_{\hat{r}_i}^\mathrm{p}=\sum_{(\boldsymbol{c},\boldsymbol{d})\in\mathcal{T}_i}\boldsymbol{w}_{(\boldsymbol{c},\boldsymbol{d})}\boldsymbol{\rho}^{\boldsymbol{c}+\boldsymbol{d}}e^{\mathrm{i}(\boldsymbol{c}-\boldsymbol{d})\cdot\boldsymbol{\theta}}    
\end{equation}
with $\mathcal{T}_i=\{(\boldsymbol{c},\boldsymbol{d}):\hat{r}_i=(\boldsymbol{c}-\boldsymbol{d})\cdot\boldsymbol{r}\}$.

In the case of Cartesian coordinates, we substitute~\eqref{eq:cartesian-form} and obtain
\begin{equation}
    \boldsymbol{q}^{\boldsymbol{c}}=(q_{1,\mathrm{s}}e^{\mathrm{i}r_1\Omega t})^{c_1}\cdots(q_{m,\mathrm{s}}e^{\mathrm{i}r_m\Omega t})^{c_m}=\boldsymbol{q}_\mathrm{s}^{\boldsymbol{c}}e^{\mathrm{i}\boldsymbol{c}\cdot\boldsymbol{r}\Omega t},
\end{equation}
\begin{align}
    \bar{\boldsymbol{q}}^{\boldsymbol{d}} & =(\bar{q}_{1,\mathrm{s}}e^{-\mathrm{i}r_1\Omega t})^{d_1}\cdots(\bar{q}_{m,\mathrm{s}}e^{-\mathrm{i}r_m\Omega t})^{d_m}\nonumber\\
    & =\bar{\boldsymbol{q}}_\mathrm{s}^{\boldsymbol{d}}e^{-\mathrm{i}\boldsymbol{d}\cdot\boldsymbol{r}\Omega t},
\end{align}
and then $\boldsymbol{q}^{\boldsymbol{c}}\bar{\boldsymbol{q}}^{\boldsymbol{d}}=\boldsymbol{q}_\mathrm{s}^{\boldsymbol{c}}\bar{\boldsymbol{q}}_\mathrm{s}^{\boldsymbol{d}}e^{\mathrm{i}(\boldsymbol{c}-\boldsymbol{d})\cdot\boldsymbol{r}\Omega t}$. Thus we have
\begin{align}
    \boldsymbol{W}(\boldsymbol{p})& =\sum_{(\boldsymbol{c},\boldsymbol{d})}\boldsymbol{w}_{(\boldsymbol{c},\boldsymbol{d})}\boldsymbol{q}^\mathrm{c}\bar{\boldsymbol{q}}^{\boldsymbol{d}}\nonumber\\
    & =\sum_{(\boldsymbol{c},\boldsymbol{d})}\boldsymbol{w}_{(\boldsymbol{c},\boldsymbol{d})}\boldsymbol{q}_\mathrm{s}^{\boldsymbol{c}}\bar{\boldsymbol{q}}_\mathrm{s}^{\boldsymbol{d}}e^{\mathrm{i}(\boldsymbol{c}-\boldsymbol{d})\cdot\boldsymbol{r}\Omega t}\nonumber\\ &=\sum_{\hat{r}_i\in\hat{\mathcal{R}}}\boldsymbol{w}_{\hat{r}_i}^\mathrm{c}e^{\mathrm{i}{\hat{r}_i}\Omega t}
\end{align}
where
\begin{equation}
\label{eq:what-ri-c}
    \boldsymbol{w}_{r_i}^\mathrm{c}=\sum_{(\boldsymbol{c},\boldsymbol{d})\in\mathcal{T}_i}\boldsymbol{w}_{(\boldsymbol{c},\boldsymbol{d})}\boldsymbol{q}_\mathrm{s}^{\boldsymbol{c}}\bar{\boldsymbol{q}}_\mathrm{s}^{\boldsymbol{d}}.
\end{equation}

With TV SSM solution used, the periodic response~\eqref{eq:ssm-time-varying} is simplified as
\begin{equation}
\boldsymbol{z}(t)=\sum_{\hat{r}_i\in\hat{\mathcal{R}}}\hat{\boldsymbol{w}}_{\hat{r}_i}e^{\mathrm{i}{\hat{r}_i}\Omega t}
\end{equation}
where $\hat{\boldsymbol{w}}_{\hat{r}_i}$ is given by~\eqref{eq:what-ri}. Thus
\begin{equation}
    A_{\mathrm{opt}}=\sum_{\hat{r}_i}\hat{\boldsymbol{w}}_{\hat{r}_i,\mathrm{opt}}e^{\mathrm{i}{\hat{r}_i}\Omega t}
\end{equation}
Note that
\begin{align}
    &\quad \int_0^T \boldsymbol{z}_I^\ast(t)\boldsymbol{Q}\boldsymbol{z}_I(t)\mathrm{d}t \nonumber\\
    & = \int_0^T\left(\sum_{\hat{r}_i}\hat{\boldsymbol{w}}_{\hat{r}_i,I}^\ast e^{-\mathrm{i}{\hat{r}_i}\Omega t}\right)\boldsymbol{Q}\left(\sum_{\hat{r}_j}\hat{\boldsymbol{w}}_{\hat{r}_j,I} e^{\mathrm{i}{\hat{r}_j}\Omega t}\right)\mathrm{d}t\nonumber\\
    & =\sum_{\hat{r}_i}\sum_{\hat{r}_j}\hat{\boldsymbol{w}}_{\hat{r}_i,I}^\ast\boldsymbol{Q}\hat{\boldsymbol{w}}_{\hat{r}_j,I}\int_0^T e^{\mathrm{i}({\hat{r}_j}-{\hat{r}_i})\Omega t}\mathrm{d}t\nonumber\\
    & = T\sum_{\hat{r}_i}\hat{\boldsymbol{w}}_{\hat{r}_i,I}^\ast\boldsymbol{Q}\hat{\boldsymbol{w}}_{\hat{r}_i,I}
\end{align}
and hence
\begin{equation}
    A_{\mathcal{L}^2}=\sqrt{\sum_{\hat{r}_i}\hat{\boldsymbol{w}}_{\hat{r}_i,I}^\ast\boldsymbol{Q}\hat{\boldsymbol{w}}_{\hat{r}_i,I}}.
\end{equation}

\section{Derivation of explicit gradients}
\label{sec:appendix-grad}
\subsection{Gradients of $A_{\mathcal{L}^2}$}
\label{sec:appendix-th4}
The derivative of $A_{\mathcal{L}^2}$ is given by
\begin{align}
    \mathcal{D}A_{\mathcal{L}^2}&=\frac{1}{2A_{\mathcal{L}^2}}\sum_{\hat{r}_i\in\hat{\mathcal{R}}}\left(\mathcal{D}\hat{\boldsymbol{w}}_{\hat{r}_i,\mathcal{I}}^\ast\boldsymbol{Q}\hat{\boldsymbol{w}}_{\hat{r}_i,\mathcal{I}}+\hat{\boldsymbol{w}}_{\hat{r}_i,\mathcal{I}}^\ast\boldsymbol{Q}\mathcal{D}\hat{\boldsymbol{w}}_{\hat{r}_i,\mathcal{I}}\right)\nonumber\\&=\frac{1}{A_{\mathcal{L}^2}}\sum_{\hat{r}_i\in\hat{\mathcal{R}}}\hat{\boldsymbol{w}}_{\hat{r}_i,\mathcal{I}}^\ast\bar{\boldsymbol{Q}}\mathcal{D}\hat{\boldsymbol{w}}_{\hat{r}_i,I}\label{eq:DA-l2}
\end{align}
where $\bar{\boldsymbol{Q}}=(\boldsymbol{Q}+\boldsymbol{Q}^\top)/2$ and (see~\eqref{eq:what-ri})
\begin{align}
\label{eq:Dwhat}
    \mathcal{D}\hat{\boldsymbol{w}}_{\hat{r}_i} = & \mathcal{D}{\boldsymbol{w}}_{\hat{r}_i}^\mathrm{cor}+\left(\delta\epsilon\boldsymbol{x}_{\boldsymbol{0}}+\epsilon\delta\boldsymbol{x}_{\boldsymbol{0}}\right)\chi_1({\hat{r}_i})+\nonumber\\& \quad\left(\delta\epsilon\bar{\boldsymbol{x}}_{\boldsymbol{0}}+\epsilon\delta\bar{\boldsymbol{x}}_{\boldsymbol{0}}\right)\chi_{-1}({\hat{r}_i})
\end{align}
Here $\chi_1(\hat{r}_i)=1$ if $\hat{r}_i=1$ and $\chi_1(\hat{r}_i)=0$ otherwise. Likewise, $\chi_{-1}(\hat{r}_i)=1$ if $\hat{r}_i=-1$ and $\chi_{-1}(\hat{r}_i)=0$ otherwise. Based on~\eqref{eq:expphi-}, we have
\begin{equation}
-\mathrm{i}\delta\Omega\boldsymbol{B}\boldsymbol{x}_{\boldsymbol{0}}+(\boldsymbol{A}-\mathrm{i}\Omega\boldsymbol{B})\delta\boldsymbol{x}_{\boldsymbol{0}}=\boldsymbol{0}.
\end{equation}
From which we obtain
\begin{equation}
\label{eq:delta-x0}
    \delta\boldsymbol{x}_{\boldsymbol{0}}=\mathrm{i}(\boldsymbol{A}-\mathrm{i}\Omega\boldsymbol{B})^{-1}\boldsymbol{B}\boldsymbol{x}_{\boldsymbol{0}}\delta\Omega.
\end{equation}
For polar coordinate representation, we have~\eqref{eq:what-ri-p} and thus
\begin{align}
\mathcal{D}\boldsymbol{w}_{\hat{r}_i}^\mathrm{p}=&\sum_{(\boldsymbol{c},\boldsymbol{d})\in\mathcal{T}_i}\boldsymbol{w}_{(\boldsymbol{c},\boldsymbol{d})}\cdot\nonumber\\&\quad\left(\mathcal{D}\boldsymbol{\rho}^{\boldsymbol{c}+\boldsymbol{d}}+\mathrm{i}\boldsymbol{\rho}^{\boldsymbol{c}+\boldsymbol{d}}(\boldsymbol{c}-\boldsymbol{d})\cdot\delta\boldsymbol{\theta}\right)e^{\mathrm{i}(\boldsymbol{c}-\boldsymbol{d})\cdot\boldsymbol{\theta}}\label{eq:Dwp}
\end{align}
where
\begin{equation}
\mathcal{D}\boldsymbol{\rho}^{\boldsymbol{c}+\boldsymbol{d}}=\boldsymbol{\rho}^{\boldsymbol{c}+\boldsymbol{d}}\sum_{j=1}^m\frac{c_j+d_j}{\rho_j}\delta\rho_j.
\end{equation}
For Cartesian coordinates, we have~\eqref{eq:what-ri-c} and thus
\begin{equation}
\label{eq:Dwc}
\mathcal{D}\boldsymbol{w}_{r_i}^\mathrm{c}=\sum_{(\boldsymbol{c},\boldsymbol{d})\in\mathcal{T}_i}\boldsymbol{w}_{(\boldsymbol{c},\boldsymbol{d})}(\mathcal{D}\boldsymbol{q}_\mathrm{s}^{\boldsymbol{c}}\bar{\boldsymbol{q}}_\mathrm{s}^{\boldsymbol{d}}+\boldsymbol{q}_\mathrm{s}^{\boldsymbol{c}}\mathcal{D}\bar{\boldsymbol{q}}_\mathrm{s}^{\boldsymbol{d}}),
\end{equation}
where
\begin{gather}
\mathcal{D}\boldsymbol{q}_\mathrm{s}^{\boldsymbol{c}}=\boldsymbol{q}_\mathrm{s}^{\boldsymbol{c}}\sum_{j=1}^m\frac{c_j}{q_{j,\mathrm{s}}}\delta(q_{j,\mathrm{s}}^\mathrm{R}+\mathrm{i}q_{j,\mathrm{s}}^\mathrm{I}),\\ \mathcal{D}\bar{\boldsymbol{q}}_\mathrm{s}^{\boldsymbol{d}}=\bar{\boldsymbol{q}}_\mathrm{s}^{\boldsymbol{d}}\sum_{j=1}^m\frac{d_j}{\bar{q}_{j,\mathrm{s}}}\delta(q_{j,\mathrm{s}}^\mathrm{R}-\mathrm{i}q_{j,\mathrm{s}}^\mathrm{I}).\label{eq:dqcs}
\end{gather}
Substitution~\eqref{eq:Dwhat},~\eqref{eq:delta-x0}-\eqref{eq:dqcs} into~\eqref{eq:DA-l2} yields the gradient below
\begin{align}
    \frac{\partial A_{\mathcal{L}^2}}{\partial\rho_j}=&\frac{1}{A_{\mathcal{L}^2}}\sum_{\hat{r}_i\in\hat{\mathcal{R}}}\hat{\boldsymbol{w}}_{\hat{r}_i,\mathcal{I}}^{\ast}\bar{\boldsymbol{Q}}\cdot\nonumber\\
&\qquad\sum_{(\boldsymbol{c},\boldsymbol{d})\in\mathcal{T}_i}\boldsymbol{w}_{(\boldsymbol{c},\boldsymbol{d})}^{\mathcal{I}}\boldsymbol{\rho}^{\boldsymbol{c}+\boldsymbol{d}}e^{\mathrm{i}(\boldsymbol{c}-\boldsymbol{d})\cdot\boldsymbol{\theta}}\frac{c_j+d_j}{\rho_j},\label{eq:Al2-rho}
\end{align}
\begin{align}
    \frac{\partial A_{\mathcal{L}^2}}{\partial\theta_j}=&\frac{\mathrm{i}}{A_{\mathcal{L}^2}}\sum_{\hat{r}_i\in\hat{\mathcal{R}}}\hat{\boldsymbol{w}}_{\hat{r}_i,\mathcal{I}}^{\ast}\bar{\boldsymbol{Q}}\cdot\nonumber\\
&\sum_{(\boldsymbol{c},\boldsymbol{d})\in\mathcal{T}_i}\boldsymbol{w}_{(\boldsymbol{c},\boldsymbol{d})}^\mathcal{I}\boldsymbol{\rho}^{\boldsymbol{c}+\boldsymbol{d}}e^{\mathrm{i}(\boldsymbol{c}-\boldsymbol{d})\cdot\boldsymbol{\theta}}(c_j-d_j),
\end{align}
\begin{align}
    \frac{\partial A_{\mathcal{L}^2}}{\partial q_{j,\mathrm{s}}^\mathrm{R}}=&\frac{1}{A_{\mathcal{L}^2}}\sum_{\hat{r}_i\in\hat{\mathcal{R}}}\hat{\boldsymbol{w}}_{\hat{r}_i,\mathcal{I}}^{\ast}\bar{\boldsymbol{Q}}\cdot\nonumber\\
&\qquad\sum_{(\boldsymbol{c},\boldsymbol{d})\in\mathcal{T}_i}\boldsymbol{w}_{(\boldsymbol{c},\boldsymbol{d})}^{\mathcal{I}}\boldsymbol{q}_\mathrm{s}^{\boldsymbol{c}}\bar{\boldsymbol{q}}_\mathrm{s}^{\boldsymbol{d}}\left(\frac{c_j}{q_{j,\mathrm{s}}}+\frac{d_j}{\bar{q}_{j,\mathrm{s}}}\right),
\end{align}
\begin{align}
    \frac{\partial A_{\mathcal{L}^2}}{\partial q_{j,\mathrm{s}}^\mathrm{I}}=&\frac{\mathrm{i}}{A_{\mathcal{L}^2}}\sum_{\hat{r}_i\in\hat{\mathcal{R}}}\hat{\boldsymbol{w}}_{\hat{r}_i,\mathcal{I}}^{\ast}\bar{\boldsymbol{Q}}\cdot\nonumber\\&\sum_{(\boldsymbol{c},\boldsymbol{d})\in\mathcal{T}_i}\boldsymbol{w}_{(\boldsymbol{c},\boldsymbol{d})}^\mathcal{I}\boldsymbol{q}_\mathrm{s}^{\boldsymbol{c}}\bar{\boldsymbol{q}}_\mathrm{s}^{\boldsymbol{d}}\left(\frac{c_j}{q_{j,\mathrm{s}}}-\frac{d_j}{\bar{q}_{j,\mathrm{s}}}\right),
\end{align}
for $j=1,\cdots,m$ and
\begin{gather}
    \frac{\partial A_{\mathcal{L}^2}}{\partial\Omega}=\frac{\epsilon}{A_{\mathcal{L}^2}}\left(\hat{\boldsymbol{w}}_{1,\mathcal{I}}^{\ast}\bar{\boldsymbol{Q}}\mathrm{i}\left((\boldsymbol{A}-\mathrm{i}\Omega\boldsymbol{B})^{-1}\boldsymbol{B}\boldsymbol{x}_{\boldsymbol{0}}\right)_\mathcal{I}+\mathrm{cc}\right),\label{eq:Al2-om}\\
    \frac{\partial A_{\mathcal{L}^2}}{\partial\epsilon}=\frac{1}{A_{\mathcal{L}^2}}\left(\hat{\boldsymbol{w}}_{1,\mathcal{I}}^{\ast}\bar{\boldsymbol{Q}}\boldsymbol{x}_{\boldsymbol{0},\mathcal{I}}+\hat{\boldsymbol{w}}_{-1,\mathcal{I}}^{\ast}\bar{\boldsymbol{Q}}\bar{\boldsymbol{x}}_{\boldsymbol{0},\mathcal{I}}\right)\label{eq:Al2-eps},
\end{gather}
where $\bar{\boldsymbol{Q}}=(\boldsymbol{Q}+\boldsymbol{Q}^\top)/2$, $\boldsymbol{w}_{(\boldsymbol{c},\boldsymbol{d})}^\mathcal{I}\in\mathbb{C}^{|\mathcal{I}|}$ extracts the corresponding entries from the vector $\boldsymbol{w}_{(\boldsymbol{c},\boldsymbol{d})}$, and {cc} denotes the corresponding complex conjugate part.

\subsection{Gradients of $A_{\mathrm{opt}}$}
\label{sec:appendix-th5}
The derivative of $A_{\mathrm{opt}}$ is given by
\begin{align}   \mathcal{D}A_{\mathrm{opt}}=&\sum_{\hat{r}_i\in\hat{\mathcal{R}}}(\mathcal{D}\hat{\boldsymbol{w}}_{\hat{r}_i,\mathrm{opt}}+\mathrm{i}\hat{r}_i\Omega\hat{\boldsymbol{w}}_{\hat{r}_i,\mathrm{opt}}\mathcal{D}t+\nonumber\\&\qquad\mathrm{i}\hat{r}_i t\hat{\boldsymbol{w}}_{\hat{r}_i,\mathrm{opt}}\mathcal{D}\Omega)e^{\mathrm{i}{\hat{r}_i}\Omega t}.\label{eq:DAopt}
\end{align}
Substitution~\eqref{eq:Dwhat},~\eqref{eq:delta-x0}-\eqref{eq:dqcs} into~\eqref{eq:DAopt} yields the gradient below
\begin{gather}
    \frac{\partial A_\mathrm{opt}}{\partial\rho_j}=\sum_{\hat{r}_i\in\hat{\mathcal{R}}}e^{\mathrm{i}\hat{r}_i\Omega t}\sum_{(\boldsymbol{c},\boldsymbol{d})\in\mathcal{T}_i}\boldsymbol{w}_{(\boldsymbol{c},\boldsymbol{d})}^{\mathrm{opt}}\boldsymbol{\rho}^{\boldsymbol{c}+\boldsymbol{d}}e^{\mathrm{i}(\boldsymbol{c}-\boldsymbol{d})\cdot\boldsymbol{\theta}}\frac{c_j+d_j}{\rho_j},\label{eq:aopt-rhoj}\\
    \frac{\partial A_\mathrm{opt}}{\partial\theta_j}={\mathrm{i}}\sum_{\hat{r}_i\in\hat{\mathcal{R}}}e^{\mathrm{i}\hat{r}_i\Omega t}\sum_{(\boldsymbol{c},\boldsymbol{d})\in\mathcal{T}_i}\boldsymbol{w}_{(\boldsymbol{c},\boldsymbol{d})}^{\mathrm{opt}}\boldsymbol{\rho}^{\boldsymbol{c}+\boldsymbol{d}}e^{\mathrm{i}(\boldsymbol{c}-\boldsymbol{d})\cdot\boldsymbol{\theta}}(c_j-d_j),\\
    \frac{\partial A_\mathrm{opt}}{\partial q_{j,\mathrm{s}}^\mathrm{R}}=\sum_{\hat{r}_i\in\hat{\mathcal{R}}}e^{\mathrm{i}\hat{r}_i\Omega t}\sum_{(\boldsymbol{c},\boldsymbol{d})\in\mathcal{T}_i}\boldsymbol{w}_{(\boldsymbol{c},\boldsymbol{d})}^{\mathrm{opt}}\boldsymbol{q}_\mathrm{s}^{\boldsymbol{c}}\bar{\boldsymbol{q}}_\mathrm{s}^{\boldsymbol{d}}\left(\frac{c_j}{q_{j,\mathrm{s}}}+\frac{d_j}{\bar{q}_{j,\mathrm{s}}}\right),\\
    \frac{\partial A_\mathrm{opt}}{\partial q_{j,\mathrm{s}}^\mathrm{I}}=\sum_{\hat{r}_i\in\hat{\mathcal{R}}}e^{\mathrm{i}\hat{r}_i\Omega t}\sum_{(\boldsymbol{c},\boldsymbol{d})\in\mathcal{T}_i}\boldsymbol{w}_{(\boldsymbol{c},\boldsymbol{d})}^{\mathrm{opt}}\boldsymbol{q}_\mathrm{s}^{\boldsymbol{c}}\bar{\boldsymbol{q}}_\mathrm{s}^{\boldsymbol{d}}\left(\frac{c_j}{q_{j,\mathrm{s}}}-\frac{d_j}{\bar{q}_{j,\mathrm{s}}}\right)
\end{gather}
for $j=1,\cdots,m$ and
\begin{align}
       \frac{\partial A_\mathrm{opt}}{\partial\Omega}=&\sum_{\hat{r}_i\in\hat{\mathcal{R}}}e^{\mathrm{i}\hat{r}_i\Omega t}\mathrm{i}\hat{r}_i t\hat{\boldsymbol{w}}_{\hat{r}_i,\mathrm{opt}}+\nonumber\\&\quad{\epsilon}\left(\mathrm{i}\left((\boldsymbol{A}-\mathrm{i}\Omega\boldsymbol{B})^{-1}\boldsymbol{B}\boldsymbol{x}_{\boldsymbol{0}}\right)_{\mathrm{opt}}e^{\mathrm{i}\Omega t}+\mathrm{cc}\right),\label{eq:Alinf-om} 
\end{align}
\begin{gather}
    \frac{\partial A_\mathrm{opt}}{\partial\epsilon}=\boldsymbol{x}_{\boldsymbol{0},\mathrm{opt}}e^{\mathrm{i}\Omega t}+\bar{\boldsymbol{x}}_{\boldsymbol{0},\mathrm{opt}}e^{-\mathrm{i}\Omega t},\label{eq:Alinf-eps}\\
    \frac{\partial A_\mathrm{opt}}{\partial t}=\sum_{\hat{r}_i\in\hat{\mathcal{R}}}e^{\mathrm{i}\hat{r}_i\Omega t}\mathrm{i}\hat{r}_i \Omega\hat{\boldsymbol{w}}_{\hat{r}_i,\mathrm{opt}}\label{eq:aopt-t},
\end{gather}
where $\boldsymbol{w}_{(\boldsymbol{c},\boldsymbol{d})}^\mathrm{opt}\in\mathbb{C}$ extracts the corresponding entry from the vector $\boldsymbol{w}_{(\boldsymbol{c},\boldsymbol{d})}$.

The derivatives of $A_{\mathcal{L}^2}$ and $A_\mathrm{opt}$ with respect to $\Omega$ involves inversion of the matrix $\boldsymbol{A}-\mathrm{i}\Omega\boldsymbol{B}\in\mathbb{C}^{N\times N}$ , as seen in~\eqref{eq:Al2-om} and~\eqref{eq:Alinf-om}. This inversion must be performed in each iteration, as the matrix is $\Omega$-dependent. Hence, the associated computational cost can be very significant if $n\gg1$. We note that this inversion is a result of the non-autonomous part of the SSM, as seen in~\eqref{eq:ssm-time-varying} and~\eqref{eq:expphi-}, and this matrix inversion is not required if we adopt the TI SSM solution~\eqref{eq:ssm-time-independent}. Thus, it is worth checking if the difference between the TV SSM solutions and the corresponding TI SSM solutions can be ignored.

% Let $\hat{\mathcal{R}}=\{\hat{r}_i\}$, we know that $r_i\in\hat{\mathcal{R}}$ for $i=1,\cdots,m$ because linear part is included in $\boldsymbol{W}(\boldsymbol{p})$. Therefore, $r_\mathrm{d}$ also defines the largest common divisor for the set of rational numbers $\hat{\mathcal{R}}$. 

\bibliography{manuscript.bbl} 
\bibliographystyle{ieeetr}

\section*{Declarations}
\section*{Conflict of interest}
The authors declare that they have no conflict of interest.

\section*{Author contributions}
All authors contributed equally to the conception of this paper. Mingwu Li was lead on the formal analysis and code development with Shobhit Jain and George Haller in supporting roles. George Haller supervised this project. The first draft of the manuscript was written by Mingwu Li and Shobhit Jain. All authors contributed to the review and editing of the text.

\section*{Data availability}
The code used to generate the numerical results included in this paper is available as part of the open-source package SSMTool 2.5 at \url{https://zenodo.org/records/10018285}.

\end{document}